\documentclass[onecolumn]{IEEEtran}
\usepackage{amsmath,amsfonts,amssymb,euscript, graphicx, 
epsfig,
enumerate,float,afterpage, subfigure, ifthen}%

\usepackage{url}
\usepackage{hyperref}

\newtheorem{thm}{Theorem}

\newtheorem{lem}{Lemma}

\newcommand{\expect}[1]{\mathbb{E}\left[#1\right]}

\newcommand{\norm}[1]{||{#1}||}

\newcommand{\script}[1]{{{\cal{#1} }}}

\newcommand{\transpose}{\top}

\begin{document}

\title
  {Fast Learning for Renewal Optimization in Online Task Scheduling}
\author{Michael J. Neely \\ University of Southern California\\ \url{https://viterbi-web.usc.edu/~mjneely/}\\
\thanks{This work was supported in part by one or more of: NSF CCF 1718477, NSF SpecEES 1824418.} 
}

\markboth{}{Neely}

\maketitle

\begin{abstract} 
This paper  considers online optimization of a renewal-reward system.  A controller performs a sequence of tasks back-to-back.  Each task has a random vector of parameters, called the \emph{task type vector}, that affects the task processing options and also affects the resulting reward and time duration of the task. 
The probability distribution for the task type vector is unknown and the controller must learn to make efficient decisions so that time average reward converges to optimality.  Prior work on such renewal optimization problems leaves open the question of optimal convergence time. This paper develops 
an algorithm with  an optimality gap that decays like $O(1/\sqrt{k})$, where $k$ is the number of tasks processed. The same algorithm is shown to have faster $O(\log(k)/k)$ performance when the system satisfies a strong concavity property.  The proposed algorithm uses an auxiliary variable that is updated according to a classic 
Robbins-Monro iteration. It makes online scheduling decisions at the start of each renewal frame based on this variable and on the observed task type. 
A matching converse is obtained for the strongly concave case by constructing 
an example system for which all algorithms have performance at best 
$\Omega(\log(k)/k)$. A matching $\Omega(1/\sqrt{k})$ converse is also shown for the general case without strong concavity.  
\end{abstract}

\section{Introduction}

Consider a system where a controller performs a sequence of tasks over time (see Fig. \ref{fig:renewal-tasks}). The tasks are performed back-to-back so that when task $k$ ends the task $k+1$ immediately begins. 
The interval of time over which the system performs task $k \in \{0, 1, 2, \ldots\}$ shall be called \emph{frame $k$}.  Fix $m$ as a positive integer. 
At the start of each frame $k \in \{0, 1, 2, \ldots\}$ the controller observes a vector $S[k] \in \mathbb{R}^m$ that determines the \emph{task type}. Components of $S[k]$ may  include parameters that determine the characteristics of
task $k$.    Assume that $\{S[k]\}_{k=0}^{\infty}$ is independent and identically distributed (i.i.d.) over frames with distribution function 
$$F_S(s) = P[S[k]\leq s] \quad \forall s \in \mathbb{R}^m$$  
where vector inequality is taken entrywise. 
The distribution function $F_S(s)$ is not necessarily known to the controller. After observing  $S[k]$, the controller chooses to operate in 
one of various \emph{task processing modes} for the duration of frame $k$.  The available modes can depend on $S[k]$.  The $S[k]$ value and the particular mode that is chosen together determine the  \emph{task duration} $T[k]$  and the \emph{task reward} $R[k]$ for frame $k$.  
For example, $R[k]$ can be the monetary profit earned by completing the task on frame $k$. 
 In a network scheduling scenario, $R[k]$ can be the total amount of data transmitted on frame $k$.  Alternatively, we can have $R[k]=-P[k]$ where $P[k]$ is a power cost incurred on frame $k$. 
 
Every frame $k \in \{0, 1, 2, \ldots\}$ the controller first observes $S[k]$ and then chooses a \emph{decision vector}: 
$$(T[k], R[k]) \in \script{D}(S[k])$$
where $\script{D}(S[k])$ is the set of all possible decision vectors available 
for the task type $S[k]$ (considering all processing modes). 
The total reward per unit time, averaged over the first $K$ tasks, is
\begin{equation} \label{eq:sample-path-intro} 
\frac{\sum_{k=0}^{K-1} R[k]}{\sum_{k=0}^{K-1} T[k]} 
\end{equation} 
The goal is to develop an algorithm for making decisions over frames that 
ensures the limiting average reward per unit time converges to optimality at the fastest possible rate. 
Long term optimality is defined by all possible algorithms, including algorithms that have knowledge of the probability distribution $F_S(s)$.  However, convergence time to optimality is considered for algorithms that have no 
a-priori knowledge of $F_S(s)$.  For fast convergence, algorithms must quickly learn whatever aspects of the distribution are relevant for making intelligent control decisions that maximize average reward. 

This problem is called a \emph{renewal optimization problem} because the system 
state renews itself on each new frame (when a new  $S[k]$ is observed).  For example, consider a stationary and randomized algorithm that, 
on every frame $k$, chooses $(T[k], R[k]) \in \script{D}(S[k])$ independently of the past using a fixed conditional probability distribution given the observed $S[k]$. Then $\{R[k]\}_{k=0}^{\infty}$ and $\{T[k]\}_{k=0}^{\infty}$ are i.i.d. and standard renewal theory implies  (see, for example, \cite{gallager}): 
$$ \lim_{K\rightarrow\infty} \frac{\sum_{k=0}^{K-1} R[k]}{\sum_{k=0}^{K-1} T[k]} = \frac{\expect{R[0]}}{\expect{T[0]}} \quad \mbox{ with prob 1} $$
In principle one could design the best stationary and randomized algorithm \emph{offline}  
by choosing, for each possible realization of $S[k]$, a conditional probability distribution 
for $(T[k], R[k])$ given $S[k]$, to maximize the following ratio of expectations: 
$$  \frac{\expect{R[k]}}{\expect{T[k]}} = \frac{\expect{\: \expect{R[k]|S[k]} \: }}{\expect{\: \expect{T[k]|S[k]} \: }} $$
This offline design would require knowledge of the probability distribution for $S[k]$. The proposed algorithm of this paper is not stationary, and so  $\{R[k]\}_{k=0}^{\infty}$ and $\{T[k]\}_{k=0}^{\infty}$ are not i.i.d. sequences. The proposed algorithm operates online with no a-priori knowledge of the distribution for $S[k]$. It must 
 adapt its decisions by learning from the past.

\subsection{Prior work} 

Optimization of renewal systems is related to linear fractional programming (see, for example, 
\cite{fractional-paper}\cite{boyd-convex}). An offline method for optimizing Markov renewal 
systems via linear fractional programming is in \cite{fox-linear-fractional-mdp}.  
Online methods for renewal optimization are developed in  \cite{renewal-opt-tac}\cite{sno-text}, which treat systems with additional time average
constraints. The work \cite{renewal-opt-tac} 
develops a drift-plus-penalty ratio rule for making decisions that are shown, over time, to satisfy the constraints and achieve a time averaged reward that is arbitrarily close to optimal. The algorithm in \cite{renewal-opt-tac} requires knowledge of the probability distribution for $S[k]$.  An approximate implementation is also given 
in \cite{renewal-opt-tac} that uses a bisection procedure that does not require knowledge of the probability distribution. This method is applied to treat data center scheduling in \cite{xiaohan-datacenter}, asynchronous renewal timelines in \cite{xiaohan-asynch-renewal}, 
power-aware computing in  \cite{neely-power-chapter}, and has connections to the delay-optimal queueing
work in \cite{chih-ping-delay-optimal-priority}.  An alternative Robbins-Monro technique is used in \cite{sno-text} and shown to perform well in simulation, but its convergence time is not anlayzed.   This prior work leaves open the question of optimal convergence time in a renewal system where there is no a-priori knowledge of the 
task type distribution $S[k]$.  That question is resolved in the current paper. 

The algorithm of the current paper is closely related to the classical Robbins-Monro iteration
\cite{robbins-monro}. The work \cite{robbins-monro} treats a 
problem of finding a root $\theta$ to an equation $M(x) =0$ in the case when a nondecreasing 
function $M:\mathbb{R}\rightarrow\mathbb{R}$
is unknown and can only be indirectly evaluated. Specifically, on each iteration $k$ we hand a value $X[k]$ to an oracle, where $X[k]$ represents our best guess of the root at time $k$, and the oracle returns
a random variable $Y[k]$ whose expectation is equal to $M(X[k])$.  That is
\begin{equation} \label{eq:Robbins}
M(x) = \expect{Y[k]|X[k]=x} \quad \forall x \in \mathbb{R}
\end{equation} 
The  estimated root is then updated via the iteration: 
\begin{equation} \label{eq:RM-it} 
X[k+1] = X[k]- \eta[k] Y[k]
\end{equation} 
where $\{\eta[k]\}_{k=0}^{\infty}$ is a sequence of positive stepsizes.   Under certain 
assumptions, the work \cite{robbins-monro} shows that  
$X[k]$  converges in probability to the root $\theta$.  The technique of \cite{robbins-monro} 
inspired the field of \emph{stochastic approximation} and has been extended to treat minimization of convex functions $M:\mathbb{R}^n\rightarrow\mathbb{R}$ when an oracle returns stochastic gradients or subgradients 
\cite{stochastic-approx-book}\cite{SGD-averaging}\cite{SGD-robust}\cite{kushner-stochastic-approx}\cite{borkar-book}. Modifications of Robbins-Monro type iterations are considered 
in \cite{proximal-robbins-monro}\cite{SGD-linear-models}, improvements for binary data are in 
\cite{robbins-monro-binary}, and applications to 
Bayesian inference is explored in \cite{robbins-monro-Bayesian}.  If $M(x)$ is 
convex then, under a carefully chosen sequence of stepsizes, 
the optimality gap decays like $O(1/\sqrt{k})$, where $k$ is the number of iterations, while if $M(x)$ is \emph{strongly} convex then the optimality gap decays like $O(1/k)$. Converse results in \cite{stochastic-approx-book} show these convergence rates cannot be improved.  However, an example in \cite{SGD-robust} shows convergence is sensitive to choosing the stepsize based on  knowledge of the strong convexity parameter. When this parameter is over-estimated the convergence can be 
as slow as $\Omega(1/k^{1/5})$. This is used to motivate robust methods in \cite{SGD-robust}.

The current paper uses an iteration similar
to \eqref{eq:RM-it}. However, renewal optimization problems  
have a different structure from the 
stochastic approximation problems described in the previous paragraph. 
For example, optimality for the current paper is related to maximizing a ratio of expectations 
$$ \frac{\expect{R[k]}}{\expect{T[k]}}$$ 
which is different from minimizing the single expectation that defines $M(x)$ in \eqref{eq:Robbins}. 
Furthermore, the expectations $\expect{R[k]}$ and $\expect{T[k]}$ are not determined by a single input parameter $x$ but by a family of conditional distributions
for choosing $(T[k], R[k])$ given $S[k]$.  
The structure of the randomness is also different: In stochastic approximation we choose a  input
vector $X[k]$ and an oracle gives us a noisy version of $M(X[k])$, whereas in the current paper the system gives us a random ``input'' state $S[k]$ from which we choose a (fully known) output $(T[k], R[k]) \in \script{D}(S[k])$.  The renewal problem is \emph{online} and all decisions (starting from frame $0$) are important in creating time averages that are close to optimal.  This is different from problems that seek a single vector $x$ and do not care what decisions are made in the past as long as they lead to a good eventual choice of $x$. Finally, there are no convexity assumptions on the sets $\script{D}(s)$.  Nevertheless, a connection to Robbins-Monro techniques is made in \cite{renewal-opt-tac} where it is shown that
the optimal time average reward $\theta^*$ is the root of the following $M(\theta)$ function
$$ M(\theta) = \expect{\sup_{(T[k], R[k]) \in \script{D}(S[k])}\left\{ R[k] - \theta T[k]\right\}}$$
where the expectation is with respect to the random $S[k]$ (that has an unknown distribution). 

The current paper uses this observation to motivate a Robbins-Monro style iteration for 
finding the root $\theta^*$ using an \emph{auxiliary variable} $\theta[k]$ on each frame $k$. However, it is not enough to simply find a value $\theta$ that is close to $\theta^*$.  That is because our goal is to obtain an online algorithm with a time average (starting from frame $0$) that converges to $\theta^*$ as quickly as possible. Fast learning is crucial because early mistakes  are included in the time average
that we are trying to optimize.    The fundamental convergence times for renewal systems that are established in this paper are different from the fundamental convergence times for stochastic approximation  in \cite{stochastic-approx-book}.  

The focus on convergence time in this paper is conceptually similar to 
analysis of regret in  multi-armed bandit systems, see for example \cite{bubeck2012regret}\cite{adversarial-bandit}\cite{auer-confidence-bandit}  and a recent renewal-based bandit formulation in 
\cite{atilla-renewal-2019}.  Such problems have an ``exploration versus exploitation'' structure that is different from the structure of the current paper.   In \cite{atilla-renewal-2019}, one of multiple ``bandit-style'' arms is pulled, each arm giving a reward after a random renewal-time. 
The goal is to learn the best arm to pull.   In contrast, rather than choosing an unknown arm and receiving a reward, our system receives a random task state $S[k]$ and chooses  
one of multiple (known) control actions $(T[k], R[k]) \in \script{D}(S[k])$ in a way that most quickly optimizes the system. 


\subsection{Our contributions} 

This paper develops an online algorithm for making decisions $(T[k], R[k]) \in \script{D}(S[k])$ on each frame $k$, without knowing the distribution for $S[k]$,  that ensures time average reward converges to the optimal value $\theta^*$ with probability 1.  The algorithm uses an auxiliary variable $\theta[k]$ that is updated according to a Robbins-Monro type iteration. However, the rate at which $\theta[k]$ converges to $\theta^*$ is faster than the rate at which the online averages converge to $\theta^*$.  Specifically, we show: 
\begin{itemize} 
\item Under general system structure, the proposed algorithm has an optimality gap
that decays like $O(1/\sqrt{k})$.

\item Under special ``strong concavity'' system structure, the proposed algorithm ensures an improved $O(\log(k)/k)$ performance.   The algorithm is robust in the sense that it 
does not require knowledge of the strong concavity parameter.  However, 
it turns out that knowledge of a minimum frame size parameter $T_{min}$ is 
crucial for selecting the stepsize. Fortunately $T_{min}$ is easily known in most practical systems. 

\item Regardless of whether or not the system exhibits strong concavity, the mean squared error between $\theta[k]$ and $\theta^*$ decays like $O(1/k)$.  It is remarkable that convergence of this auxiliary variable is independent of strong concavity and is fundamentally different from convergence of the online time averages (which do depend on strong concavity structure). 

\item We present a matching $\Omega(\log(k)/k)$ converse result for systems with strongly 
concave structure. This is done by constructing an example system for which no algorithm can achieve faster convergence.  The proof utilizes a Bernoulli estimation theorem from \cite{hazan-kale-stochastic} and a recent mapping technique in \cite{neely-NUM-converse-infocom2020}.
 
\item We also present a matching $\Omega(1/\sqrt{k})$ converse for an example system without the strongly concave property.  The proof has a different structure from the strongly concave case: It shows that for a given $\epsilon>0$ and for any algorithm operating on the system, there is a probability parameter for $S[k]$ under which the algorithm needs at least $\Omega(1/\epsilon^2)$ frames to ensure performance is within $\epsilon$ of optimal.  This result is similar in spirit to a known square-root converse result for pseudo-regret in 
multi-arm bandit problems in  \cite{bubeck2012regret}.  However,  the square root law does not arise for the same reason as in \cite{bubeck2012regret}. Indeed, the two problems are structurally different and use different proofs.  
\end{itemize}

\section{Examples}  \label{section:example}

\subsection{Cloud computing with two choices} 

\begin{figure}[htbp]
   \centering
   \includegraphics[width=6in]{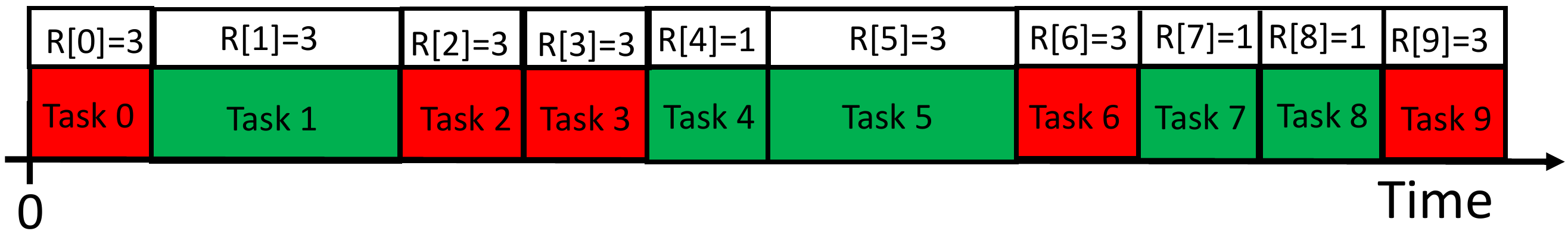} 
   \caption{A sequence of back-to-back tasks of different types:  Red tasks ($S[k]=0$) have  $(T[k], R[k])=(1,3)$.  Green tasks ($S[k]=1$) have   $(T[k], R[k]) \in \{(1, 1), (2, 3)\}$.}
   \label{fig:renewal-tasks}
\end{figure}

Suppose a cloud computing device performs back-to-back tasks. Each task is one of two types, shown as red or green in Fig. \ref{fig:renewal-tasks}. 
Let $S[k] \in \{0,1\}$ be the type of task $k$ and assume $\{S[k]\}_{k=0}^{\infty}$ is i.i.d. with $P[S[k]=1]=p$.  
Red tasks always take 1 unit of time and earn 3 units of revenue.  Green 
tasks earn revenue depending on the \emph{quality} of the processing, and there are only two processing modes available. The decision sets are described in the table below:

\begin{table}[h] 
\begin{tabular}{|c|c|c|} 
\hline
\mbox{Task type} & \mbox{color} & \mbox{Decision set}  \\ \hline \hline
$S[k]=0$ & red & $\script{D}(0) = \{(1,3)\}$\\ \hline
$S[k]=1$ & green & $\script{D}(1) = \{(2,3), (1,1)\}$ \\ \hline
\end{tabular} 
\end{table}

If $S[k]=1$ we can choose either high quality processing (which takes 2 units of time and brings revenue 3) or low quality processing (which takes 1 unit of time and brings revenue 1). Which should we choose?  Since $\frac{3}{2} > \frac{1}{1}$, a naive guess is that it is always optimal to choose high quality.   This is not true: It depends on the value of $p$. This is because Type 0 tasks are more valuable than Type 1 tasks, so it may be better to quickly get a Type 1 task over with in hopes that the next task is Type 0. 
It can be shown that it is best to always choose low quality if $0\leq p < 1/2$, and to always choose high quality if $1/2\leq p \leq 1$.  If $p$ is unknown, 
an intuitively good online control algorithm  
is to form a running estimate $\hat{p}[k] = \frac{1}{k}\sum_{j=0}^{k-1} 1_{\{S[j]=1\}}$ and choose high quality whenever $\hat{p}[k]\geq 1/2$.  Unfortunately this method can make mistakes if $p \approx 1/2$.  Section \ref{section:square-root} shows that no causal algorithm that does not have a-priori knowledge of $p$ can optimize this system faster than a  square root law.

\subsection{Infinitely many choices} 

Now suppose the time $T[k]$ spent on each task can be chosen as any real number in the interval 
$[T_{min}, T_{max}]$, where $T_{min}$ and $T_{max}$ are some given positive constants.  Suppose the 
corresponding revenue is: 
$$ R[k] = A[k]f(T[k]B[k], C[k])$$ 
where $f:\mathbb{R}^2\rightarrow \mathbb{R}$ is some function (possibly discontinuous and nonconvex) that is nondecreasing in the first coordinate. Thus, revenue can be larger if more time in spent on the task.  The peculiarities of each task are described by the random vector $S[k]=(A[k], B[k], C[k])$, which has an unknown joint cumulative distribution $F_{A,B,C}(a,b,c)$.  Every frame $k$ the controller observes $S[k]$ and chooses $T[k] \in [T_{min}, T_{max}]$.

\subsection{Project selection} \label{section:selection} 

Suppose Alice  works on one project at a time. 
On each frame $k$, Alice receives 
a random number $N[k]$ of new potential projects, each with a different profit and 
time commitment.  Suppose $N[k]$ has some unknown probability mass function (for example, $N[k]$ might be a Poisson random variable 
with some unknown parameter $\lambda$).  
Let $(T_j[k], R_j[k])$ represent the time and profit characteristics for each project $j \in \{1, \ldots, N[k]\}$. The joint distribution for these random
parameters is arbitrary and unknown. 
At the start of frame $k$, Alice chooses 
which single project $j \in \{1, \ldots, N[k]\}$ to work on.  She can also 
choose to work on nothing for one time unit, which yields $(T[k],R[k]) = (1,0)$.
This is useful if she believes that  none of the current project options are desirable.  For example, if there is only one project option and it requires a large time commitment but brings only a small profit, she might reject this project 
in favor of waiting for a new batch of options.   If $N[k]=0$ then there are no options and so $(T[k], R[k])=(1,0)$.

The task type $S[k]$ formally contains all $N[k]$ and $(T_j[k], R_j[k])$ parameters, and the decision set 
is:\footnote{There is no loss of generality in assuming $S[k] \in \mathbb{R}^m$ for some positive integer $m$. 
That is because we can formally pack an infinite sequence of task parameters into a single real number by organizing the digits of its infinite decimal expansion.} 
$$ \script{D}(S[k]) = \{(1,0), (T_1[k], R_1[k]), \ldots, (T_{N[k]}[k], \ldots, R_{N[k]}[k])\}$$
How do we know if a particular option, say $(10.7, 1.2)$, is good or bad? Can we learn to make good project selection decisions?  How fast can we learn?

\section{Formulation} 

As described in the introduction, the task type sequence $\{S[k]\}_{k=0}^{\infty}$ is i.i.d. over frames.  Every frame $k \in \{0, 1,2, \ldots\}$ the controller observes $S[k]$ and chooses $(T[k], R[k]) \in \script{D}(S[k])$, where $\script{D}(S[k])$ is a known set of options for $(T[k], R[k])$ that are available  under  $S[k]$. 

\subsection{Structural assumptions} \label{section:structure-assumptions} 

Assume $\{S[k]\}_{k=0}^{\infty}$ is an i.i.d. sequence of random vectors that take values in a set $\Omega_S\subseteq\mathbb{R}^m$. For each $s \in \Omega_S$ the decision set $\script{D}(s)$ is assumed to be a compact subset of $\mathbb{R}^2$ that satisfies
$$ \script{D}(s) \subseteq (0, \infty) \times (-\infty, \infty)$$
This ensures all vectors $(t,r) \in \script{D}(s)$ have $t>0$ (so frame sizes are  positive). 
It is assumed that on each slot $k$ the decisions $(T[k], R[k]) \in \script{D}(S[k])$ 
are made according to some probability law that ensures $T[k]$ and $R[k]$ are both random variables, 
meaning they are \emph{probabilistically measurable}.\footnote{The assumption that $T[k]$ and $R[k]$ are probabilistically measurable is mild and  precludes using the Axiom of Choice to make  nonmeasurable decisions.}  The probability law can be different for each frame and can depend on
 observations from previous frames. 

It is useful to ensure bounds on the first and second moments of $T[k]$ and $R[k]$ that
hold regardless of the decisions. 
For this we 
assume existence of  lower bound and upper bound  functions $L:\Omega_S\rightarrow [0, \infty)$ and  $U:\Omega_S \rightarrow [0, \infty)$ that satisfy 
\begin{align}
&\inf_{(t,r) \in \script{D}(s)} \{t\} \geq L(s) \quad \forall s \in \Omega_S \label{eq:bounding1} \\
&\sup_{(t,r) \in \script{D}(s)}\{ t^2+r^2\} \leq U(s) \quad \forall s \in \Omega_S \label{eq:bounding2} 
\end{align}
and for which $L(S[0])$ and $U(S[0])$ are random variables with expectations that 
satisfy 
\begin{align}
\expect{L(S[0])}&>0 \label{eq:bounding3} \\
\expect{U(S[0])} &<\infty \label{eq:bounding4} 
\end{align}
An example when such bounding functions exist is when  $\script{D}(s) \subseteq [t_1, t_2] \times [r_1, r_2]$ for all $s \in \Omega_S$, where $t_1, t_2$ and $r_1, r_2$ are some constants that satisfy $0<t_1\leq t_2$ and $r_1\leq r_2$.  In that case we have constant functions $L(s) = t_1$ and $U(s) = t_2^2+\max\{r_1^2,r_2^2\}$ for all $s \in \Omega_S$. 

From \eqref{eq:bounding1}-\eqref{eq:bounding4} it follows that regardless of the decisions we have 
\begin{align*}
T[k] &\geq L(S[k])  \quad \forall k \in \{0, 1, 2, \ldots\}  \\
T[k]^2 + R[k]^2 &\leq U(S[k])   \quad \forall k \in \{0, 1, 2, \ldots\}
\end{align*}
and there are constants $T_{min}$, $T_{max}$, $R_{min}$, $R_{max}$, $C_1$, $C_2$ 
such that for all $k \in \{0, 1, 2, \ldots\}$: 
\begin{align}
&T_{min}>0 \label{eq:moment1} \\
&T_{min} \leq \expect{T[k]}\leq T_{max}\label{eq:moment2} \\
&R_{min} \leq \expect{R[k]} \leq R_{max} \label{eq:moment3} \\
&\expect{T[k]^2} \leq C_1\label{eq:moment4} \\
&\expect{R[k]^2} \leq C_2 \label{eq:moment5}
\end{align}
It is assumed the controller knows the values of $T_{min}, T_{max}, R_{min}, R_{max}$.

\subsection{Optimization goal} 

The goal is to choose decision vectors over time  to solve: 
\begin{align} 
\mbox{Maximize:} \quad & \liminf_{K\rightarrow\infty} \frac{\sum_{k=0}^{K-1}\expect{R[k]}}{\sum_{k=0}^{K-1} \expect{T[k]}}\label{eq:goal1} \\
\mbox{Subject to:} \quad & (T[k],R[k]) \in \script{D}(S[k]) \quad \forall k \in \{0, 1, 2, \ldots\} \label{eq:goal2} 
\end{align} 
The objective  in \eqref{eq:goal1} considers a ratio of expectations, similar to the renewal optimization problems considered in 
\cite{renewal-opt-tac}\cite{sno-text}. Define $\theta^*$ as the supremum value of the objective function \eqref{eq:goal1} over all algorithms that satisfy the constraints 
\eqref{eq:goal2}. The value $\theta^*$ considers  
all algorithms that result in measurable decision vectors, 
including algorithms that know, starting from time $0$, the full distribution function $F_S(s)$ and all future values $\{S[k]\}_{k=0}^{\infty}$.  
The key result of this paper is to establish the fundamental \emph{convergence time} required to approach a value close to 
$\theta^*$ under the more practical class of algorithms that are \emph{causal} (so they have no knowledge of the future) and \emph{statistics unaware} (so they have no a-priori knowledge of the 
distribution $F_S(s)$).

\subsection{Characterizing optimality}

This subsection summarizes key facts from \cite{renewal-opt-tac}. Let $\script{A}\subseteq [T_{min}, T_{max}] \times [R_{min}, R_{max}]$ 
be the set of all 1-shot expectations $(\expect{T[0]}, \expect{R[0]})$ achievable on frame $0$, considering all possible conditional probability distributions for choosing $(T[0], R[0]) \in \script{D}(S[0])$ given the observed task type $S[0]$.  It can be shown that $\script{A}$ is a convex set. Because 
$\{S[k]\}_{k=0}^{\infty}$ is i.i.d. over frames, the set of expectations achievable on slot $0$ is the same as the set of expectations achievable on any slot $k$. Thus, under any algorithm for making probabilistically measurable decisions over frames we have 
$$ (\expect{T[k]}, \expect{R[k]}) \in \script{A} \quad \forall k \in \{0, 1, 2, \ldots\} $$
and since a convex combination of vectors in the convex set $\script{A}$ must also be 
in $\script{A}$, we have 
$$ \frac{1}{K}\sum_{k=0}^{K-1} (\expect{T[k]}, \expect{R[k]})  \in \script{A} \quad \forall K \in \{1, 2, 3, \ldots\}$$
Let $\overline{\script{A}}$ denote the closure of set $\script{A}$.  Recall that all points $(t,r) \in \overline{\script{A}}$ have $t\geq T_{min}>0$.  It can be shown that the supremum objective $\theta^*$ for problem \eqref{eq:goal1}-\eqref{eq:goal2}  is achievable by a particular (possibly non-stationary) 
algorithm for making decisions over frames  and satisfies:
\begin{equation}\label{eq:theta-star}
 \theta^*= \sup_{(t,r) \in \overline{\script{A}}} \left\{\frac{r}{t}   \right\}
\end{equation} 
The supremum on the right-hand-side of \eqref{eq:theta-star} 
is achievable because $r/t$ is a continuous function over the compact set $\overline{\script{A}}$. 
In particular, there exists a (possibly non-unique) optimal point $(t^*, r^*) \in \overline{\script{A}}$ such that 
\begin{equation*} 
\theta^* = \frac{r^*}{t^*}
\end{equation*} 

\begin{lem} \label{lem:y-theta-zero}  (From \cite{renewal-opt-tac}) Let $\theta^*$ be the optimal ratio in \eqref{eq:theta-star}. Then
\begin{equation} \label{eq:sup-is-zero}
\sup_{(t,r) \in \overline{\script{A}}} \{r-\theta^*t \} = 0
\end{equation} 
\end{lem}

\section{An iterative algorithm} 

Here we develop an algorithm that is causal and that does not know the distribution function $F_S(s)$. 
Assume the values $T_{min}$, $T_{max}$, $R_{min}$, $R_{max}$ are 
known to the controller. Let $\theta_{min}$ and $\theta_{max}$ be finite and known values that satisfy:
\begin{equation} \label{eq:theta-bracket} 
\theta_{min} \leq\theta^* \leq \theta_{max}
 \end{equation} 
For example, the following values for $\theta_{min}$ and $\theta_{max}$ can be used: 
\begin{align*}
\theta_{min} &= \min\left\{\frac{R_{min}}{T_{min}}, \frac{R_{min}}{T_{max}}\right\}\\
\theta_{max} &= \max\left\{\frac{R_{max}}{T_{min}}, \frac{R_{max}}{T_{max}}\right\}
\end{align*}
Values of $\theta_{min}$ and $\theta_{max}$ that more tightly bracket $\theta^*$ can be used
if the structure of the system allows such tighter values  
to be known. 
The following algorithm introduces a sequence of \emph{estimates} $\{\theta[k]\}_{k=0}^{\infty}$ of   $\theta^*$ as follows: Initialize $\theta[0] \in [\theta_{min}, \theta_{max}]$ as an arbitrary deterministic value.  
On each frame $k \in \{0, 1, 2, \ldots\}$ do: 
\begin{itemize} 
\item Observe $S[k] \in \Omega_S$ and the current $\theta[k]$ value. Choose $(T[k], R[k])$ to solve:
\begin{align} 
\mbox{Maximize:} \quad & R[k]-\theta[k]T[k] \label{eq:alg1a} \\
\mbox{Subject to:} \quad & (T[k], R[k]) \in \script{D}(S[k]) \label{eq:alg1b} 
 \end{align} 
breaking ties arbitrarily.\footnote{As a minor detail, we note that the tiebreaking rule must conform to a particular probability law that yields probabilistically measurable $(T[k], R[k])$ variables.   An example tiebreaking rule is to choose, among all vectors $(t, r) \in \script{D}(S[k])$ that tie, the vector with the smallest $t$-coordinate.}  There is at least one maximizer because the set 
$\script{D}(S[k])$ is compact. 
\item Update $\theta[k]$ via the iteration: 
\begin{equation} \label{eq:theta-update} 
\theta[k+1] = \left[ \theta[k] + \eta[k](R[k]-\theta[k]T[k])\right]_{\theta_{min}}^{\theta_{max}} 
\end{equation} 
where $\{\eta[k]\}_{k=0}^{\infty}$ is a deterministic 
sequence of \emph{step sizes} to be chosen later;  $[x]_{\theta_{min}}^{\theta_{max}}$ denotes a projection of real number $x$ onto the interval $[\theta_{min}, \theta_{max}]$. 
\end{itemize} 

The update equation \eqref{eq:theta-update} is inspired by the classic 
Robbins-Monro iteration in \eqref{eq:RM-it}. 

\subsection{Technical lemma} \label{section:tech-intro}

\begin{lem} \label{lem:tech-intro}  Consider algorithm \eqref{eq:alg1a}-\eqref{eq:theta-update} with any 
$\theta[0] \in [\theta_{min}, \theta_{max}]$ and any positive stepsizes $\{\eta[k]\}_{k=0}^{\infty}$. 
For each $k \in \{0, 1, 2, \ldots\}$ and for any versions of $\expect{T[k]|\theta[k]}$ and $\expect{R[k]|\theta[k]}$, the following holds with probability 1:\footnote{Recall that if $X$ and $Y$ are two random variables with $\expect{|X|}<\infty$ then: (i) There can be multiple \emph{versions} of the conditional expectation $\expect{X|Y}$; (ii) Each version is a random variable that is a deterministic function of $Y$; (iii) Any two versions $\phi(Y)$ and $\psi(Y)$ satisfy $P[\phi(Y)=\psi(Y)]=1$. }
\begin{equation} \label{eq:intro-inclusion} 
 \left(\expect{T[k]|\theta[k]}, \expect{R[k]|\theta[k]}\right) \in \overline{\script{A}}  
 \end{equation} 
That is, \eqref{eq:intro-inclusion}  holds for \emph{almost all realizations of $\theta[k]$}.   Further, there are versions of   $\expect{T[k]|\theta[k]}$ and $\expect{R[k]|\theta[k]}$ under which \eqref{eq:intro-inclusion}  holds \emph{surely for all realizations} of $\theta[k]$. 
\end{lem} 

\begin{proof} 
This is a special case of a more general result developed in the Appendix. 
\end{proof} 

Intuitively, the above lemma holds because $\{S[k]\}_{k=0}^{\infty}$ is i.i.d. over frames.  Since $\theta[k]$ depends only on 
$\{S[0], \ldots, S[k-1]\}$,  the random variables $\theta[k]$ and $S[k]$ are independent.  
Since the decision set $\script{D}(S[k])$ depends only on $S[k]$, knowing $\theta[k]$ does not change the set of expectations that can be achieved on frame $k$. For the rest of this paper  $\expect{T[k]|\theta[k]}$ and $\expect{R[k]|\theta[k]}$ shall represent any particular versions of the conditional expectations. 

\subsection{Analysis of the $(T[k], R[k])$ decision} 

The following lemmas use $\theta^*$ as the optimal ratio in \eqref{eq:theta-star} and assume
$(t^*, r^*)$ is a vector in $\overline{\script{A}}$ that satisfies $\theta^*=r^*/t^*$.  

\begin{lem} \label{lem1} Consider the algorithm \eqref{eq:alg1a}-\eqref{eq:theta-update} with any 
$\theta[0] \in [\theta_{min}, \theta_{max}]$ and any positive stepsizes $\{\eta[k]\}_{k=0}^{\infty}$.  For each $k \in \{0, 1, 2, \ldots\}$ we have for \emph{almost all realizations} of $\theta[k]$:
\begin{equation} \label{eq:R-theta-bound} 
\expect{R[k] - \theta[k]T[k]|\theta[k]}\geq t^*(\theta^*-\theta[k]) 
\end{equation} 
and 
\begin{equation} \label{eq:R-theta-bound2} 
\expect{R[k]-\theta^*T[k]|\theta[k]} \geq (\theta[k]-\theta^*)\expect{(T[k]-t^*)|\theta[k]}  
\end{equation} 
That is, the probability that random variable $\theta[k]$ does not satisfy both \eqref{eq:R-theta-bound} and \eqref{eq:R-theta-bound2} is $0$.
\end{lem} 

\begin{proof} 
Fix $k \in \{0, 1, 2, \ldots\}$.  The decision vector $(T[k], R[k])$ 
is chosen by observing $\theta[k]$ and $S[k]$ and selecting the vector in $\script{D}(S[k])$ that maximizes 
$R[k]-\theta[k]T[k]$, and so 
$$ R[k]-\theta[k]T[k] \geq R^*[k] - \theta[k]T^*[k]$$
where $(T^*[k], R^*[k])$ is any other (potentially randomized) vector in $\script{D}(S[k])$. Taking conditional expectations gives (with probability 1):\footnote{This uses the measure theory 
fact that if $X$ and $Y$ are random variables with finite expectations that satisfy $X-Y\geq 0$ surely,  and if $Z$ is another random variable, then $\expect{X-Y|Z}\geq 0$ with probability 1.} 
\begin{equation} \label{eq:key-compare} 
\expect{R[k]-\theta[k]T[k]|\theta[k]}  \geq \expect{R^*[k]|\theta[k]} - \theta[k]\expect{T^*[k]|\theta[k]} 
\end{equation} 
Note that  $\theta[k]$ depends on history in the system that occurred before frame $k$, and in particular $\theta[k]$ is independent of $S[k]$. Fix $(t, r) \in \script{A}$. By definition of $\script{A}$, there is a conditional distribution for choosing $(T[0], R[0]) \in \script{D}(S[0])$, given the observed $S[0]$, such that 
$(\expect{T[0]}, \expect{R[0]}) = (t,r)$. Since $S[k]$ is independent of $\theta[k]$ and has the same distribution as $S[0]$, we can use the same conditional distribution to produce a random vector  $(T^*[k], R^*[k]) \in \script{D}(S[k])$ 
that is independent of $\theta[k]$ such that 
$$(\expect{T^*[k]}, \expect{R^*[k]}) = (t,r)$$
and since $(T^*[k], R^*[k])$ is independent of $\theta[k]$ we have (with probability 1): 
\begin{align*}
&\expect{T^*[k]|\theta[k]} = t\\
&\expect{R^*[k]|\theta[k]} = r
\end{align*} 
Substituting these identities into \eqref{eq:key-compare}
gives (with probability 1): 
\begin{equation}
\expect{R[k]-\theta[k]T[k]|\theta[k]} \geq r - \theta[k] t \label{eq:above-inequality} 
\end{equation}
Let $(t^*, r^*)$ be a vector in $\overline{\script{A}}$ that satisfies $\theta^*=r^*/t^*$.   Since $(t^*, r^*)$ is in the \emph{closure} of the set $\script{A}$, there is a sequence of points  $\{(t_i, r_i)\}_{i=1}^{\infty}$ in $\script{A}$ that 
converge to the value $(t^*, r^*)$.  Since \eqref{eq:above-inequality} was shown to  hold (with probability 1) for an arbitrary $(t,r) \in \script{A}$, with probability $1$  it holds  simultaneously 
for all of the (countably many) $(t_i, r_i) \in \script{A}$ for $i \in \{1, 2, 3, \ldots\}$ and so:\footnote{Recall that if $\{F_i\}_{i=1}^{\infty}$ is an infinite sequence of events that satisfy $P[F_i]=1$ for all $i \in \{1, 2, 3, \ldots\}$ then $P[\cap_{i=1}^{\infty} F_i] = 1$.} 
$$ \expect{R[k]-\theta[k]T[k]|\theta[k]} \geq r_i - \theta[k]t_i \quad \forall i \in \{1, 2, 3, \ldots\}$$
Taking a limit as $i\rightarrow \infty$ yields (with probability $1$):
\begin{align*}
\expect{R[k]-\theta[k]T[k]|\theta[k]} &\geq r^* - \theta[k] t^*  \\
&= t^*(\theta^*-\theta[k]) 
\end{align*}
where the final equality uses $\theta^*=r^*/t^*$.  This proves \eqref{eq:R-theta-bound}.  
Adding  $(\theta[k]-\theta^*)\expect{T[k]|\theta[k]}$ to both sides proves  \eqref{eq:R-theta-bound2}.  
\end{proof}

\begin{lem} \label{lem2} Consider the algorithm \eqref{eq:alg1a}-\eqref{eq:theta-update} with any 
$\theta[0] \in [\theta_{min}, \theta_{max}]$ and any positive stepsizes $\{\eta[k]\}_{k=0}^{\infty}$. We have for all positive integers $K$: 
\begin{equation} \label{eq:abs-val} 
 \left|\theta^*-\frac{\sum_{k=0}^{K-1} \expect{R[k]}}{\sum_{k=0}^{K-1}\expect{T[k]}}\right| \leq  \frac{1}{KT_{min}}\sum_{k=0}^{K-1} \expect{|\theta[k]-\theta^*|\cdot|T[k]-t^*|}
 \end{equation} 
 where $T_{min}$ satisfies \eqref{eq:moment1}-\eqref{eq:moment2}.
\end{lem} 

\begin{proof} 
Fix $k \in \{0, 1, 2, \ldots\}$.  We have $(\expect{T[k]}, \expect{R[k]}) \in \overline{\script{A}}$ 
and so 
\begin{align*}
\expect{R[k]} - \theta^*\expect{T[k]} &\leq \sup_{(t, r) \in \overline{\script{A}}}[r-\theta^* t] = 0 
\end{align*}
where the final equality holds by Lemma \ref{lem:y-theta-zero}.  Summing over $k \in \{0, \ldots, K-1\}$ gives
$$ \sum_{k=0}^{K-1} \expect{R[k] - \theta^*T[k]} \leq 0 $$
Rearranging terms yields
\begin{equation} \label{eq:part1}
 \theta^* \geq \frac{\sum_{k=0}^{K-1} \expect{R[k]}}{\sum_{k=0}^{K-1} \expect{T[k]}}
 \end{equation}

On the other hand, we can  take expectations of \eqref{eq:R-theta-bound2} and use the law of iterated expectations to obtain 
$$ \expect{R[k]-\theta^*T[k]} \geq \expect{(\theta[k]-\theta^*)(T[k]-t^*)} $$
Summing this over $k \in \{0, \ldots, K-1\}$ and rearranging terms yields
\begin{align}
\frac{\sum_{k=0}^{K-1} \expect{R[k]}}{\sum_{k=0}^{K-1} \expect{T[k]}} &\geq \theta^* + \frac{1}{\sum_{k=0}^{K-1}\expect{T[k]}}\sum_{k=0}^{K-1}\expect{(\theta[k]-\theta^*)(T[k]-t^*)} \nonumber \\
&\geq \theta^* - \frac{1}{KT_{min}}\sum_{k=0}^{K-1} \expect{|\theta[k]-\theta^*|\cdot|T[k]-t^*|} \label{eq:part2}
\end{align}
Combining \eqref{eq:part1} and \eqref{eq:part2} proves the lemma. 
\end{proof}

To make the right-hand-side of \eqref{eq:abs-val} 
close to $0$, it is desirable to make the $\theta[k]$ parameter close to $\theta^*$.

\subsection{Analysis of the update rule} 

Consider the algorithm \eqref{eq:alg1a}-\eqref{eq:theta-update} with any $\theta[0]\in [\theta_{min}, \theta_{max}]$ and any positive stepsizes. Define $b$ as a constant that satisfies 
\begin{equation} \label{eq:b}
\frac{1}{2}\expect{(R[k]-\theta[k]T[k])^2} \leq b \quad \forall k \in \{0, 1, 2, \ldots\}
\end{equation} 
Such a constant $b$ exists because $\theta[k] \in [\theta_{min}, \theta_{max}]$ and 
the $R[k]$ and $T[k]$ variables satisfy the first and second moment bounds \eqref{eq:moment1}-\eqref{eq:moment5}.  The following lemma is  similar in spirit to the analysis of Robbins-Monro iterations for different systems
in \cite{robbins-monro}\cite{SGD-robust}.

\begin{lem}  Under algorithm \eqref{eq:alg1a}-\eqref{eq:theta-update} with any $\theta[0] \in [\theta_{min}, \theta_{max}]$ and  any positive stepsizes $\{\eta[k]\}_{k=0}^{\infty}$, we have for all frames $k \in \{0, 1, 2, \ldots\}$
\begin{equation} \label{eq:iteration} 
\frac{1}{2}\expect{(\theta[k+1]-\theta^*)^2} \leq  \left(\frac{1}{2}-T_{min}\eta[k]\right)\expect{(\theta[k]-\theta^*)^2} + \eta[k]^2b
\end{equation} 
where the constant $b$ satisfies \eqref{eq:b}; the constant  $T_{min}$ satisfies \eqref{eq:moment1}-\eqref{eq:moment2}; the constant $\theta^*$ is defined by  \eqref{eq:theta-star}. 
\end{lem} 

\begin{proof}  Fix $k \in \{0, 1, 2, \ldots\}$. 
For simplicity of notation define $z[k] = \theta[k]+\eta[k](R[k]-\theta[k]T[k])$. 
Recall that $[x]_{\theta_{min}}^{\theta_{max}}$ denotes a projection 
of the real number $x$ onto the interval $[\theta_{min}, \theta_{max}]$.  
By \eqref{eq:theta-update} we have 
\begin{align*}
(\theta[k+1]-\theta^*)^2 &= ([z[k]]_{\theta_{min}}^{\theta_{max}} - \theta^*)^2\\
&\overset{(a)}{=} ([z[k]]_{\theta_{min}}^{\theta_{max}} - [\theta^*]_{\theta_{min}}^{\theta_{max}})^2\\
&\overset{(b)}{\leq}  (z[k]-\theta^*)^2
\end{align*}
where (a) uses the fact that $\theta^* \in [\theta_{min}, \theta_{max}]$; (b) uses the fact that the distance between the projections of two real numbers onto a closed interval is less than or equal to the distance between those real numbers. 
Thus
\begin{align*}
\frac{1}{2}(\theta[k+1]-\theta^*)^2 &\leq \frac{1}{2}(\theta[k] - \theta^* + \eta[k](R[k]-\theta[k]T[k]))^2\\
&=\frac{1}{2}(\theta[k]-\theta^*)^2 + \frac{\eta[k]^2}{2}(R[k]-\theta[k]T[k])^2 + \eta[k](\theta[k]-\theta^*)(R[k]-\theta[k]T[k])
\end{align*}
Taking expectations gives 
\begin{equation} \label{eq:sub1} 
 \frac{1}{2}\expect{(\theta[k+1]-\theta^*)^2} \leq \frac{1}{2}\expect{(\theta[k]-\theta^*)^2}  + \eta[k]^2b + \eta[k]\expect{(\theta[k]-\theta^*)(R[k]-\theta[k]T[k])} 
 \end{equation}  
To complete the proof it suffices to provide the following bound on the final term of \eqref{eq:sub1}:  
$$\eta[k]\expect{(\theta[k]-\theta^*)(R[k]-\theta[k]T[k])}  \leq -\eta[k]T_{min}\expect{(\theta[k]-\theta^*)^2} 
$$
To do this, it suffices to show the following conditional expectation holds for almost all  realizations of $\theta[k]$, that is with probability 1: 
\begin{equation} \label{eq:suffices} 
 (\theta[k]-\theta^*)\expect{R[k]-\theta[k]T[k]|\theta[k]} \leq -T_{min}(\theta[k]-\theta^*)^2 
\end{equation} 
To show \eqref{eq:suffices} we consider two cases. 

\begin{itemize} 
\item Case 1: Suppose $\theta[k]-\theta^*<0$.  By \eqref{eq:R-theta-bound} we have for almost
all $\theta[k]$ for which $\theta[k]-\theta^*<0$:
$$\expect{R[k]-\theta[k]T[k]|\theta[k]} \geq t^*(\theta^*-\theta[k]) $$
Multiplying both sides by the (negative) value $\theta[k]-\theta^*$ flips the inequality  to yield 
\begin{align*}
(\theta[k]-\theta^*)\expect{R[k]-\theta[k]T[k]|\theta[k]} &\leq -t^*(\theta^*-\theta[k])^2\\
&\leq  -T_{min}(\theta^*-\theta[k])^2
\end{align*}
where the final inequality holds because $t^*\geq T_{min}$.  Thus, \eqref{eq:suffices} holds in this Case 1. 

\item Case 2: Suppose $\theta[k]-\theta^* \geq 0$.  We have
\begin{align*}
(\theta[k]-\theta^*)(R[k]-\theta[k]T[k]) &=(\theta[k]-\theta^*)(R[k]-\theta^*T[k]) - T[k](\theta[k]-\theta^*)^2 
\end{align*}
Taking conditional expectations given $\theta[k]$ gives, for almost all $\theta[k]$ that satisfy 
$\theta[k]-\theta^*\geq 0$: 
\begin{align}
(\theta[k]-\theta^*)\expect{R[k]-\theta[k]T[k]|\theta[k]} &= (\theta[k]-\theta^*)\expect{R[k]-\theta^*T[k]|\theta[k]} - (\theta[k]-\theta^*)^2\expect{T[k]|\theta[k]} \nonumber \\
&\leq (\theta[k]-\theta^*)\expect{R[k]-\theta^*T[k]|\theta[k]} - T_{min}(\theta[k]-\theta^*)^2 \label{eq:underbrace} 
\end{align}
where the final inequality holds because $\expect{T[k]|\theta[k]} \geq T_{min}$ (recall Lemma \ref{lem:tech-intro}).  
However, by Lemma \ref{lem:tech-intro} we know  the conditional 
expectations $(\expect{T[k]|\theta[k]}, \expect{R[k]|\theta[k]})$ are in the set $\overline{\script{A}}$ (with probability 1) and so (with probability 1): 
\begin{align*}
\expect{R[k]|\theta[k]}- \theta^* \expect{T[k] | \theta[k]} &\leq \sup_{(t,r) \in \overline{\script{A}}}\left[ r -\theta^* t\right] \\
&= 0
\end{align*}
where the final equality holds by Lemma \ref{lem:y-theta-zero}. Multiplying the above inequality by the nonnegative value $(\theta[k]-\theta^*)$ does not flip the inequality
and we obtain 
$$ (\theta[k]-\theta^*)\expect{R[k]-\theta[k]T[k]|\theta[k]} \leq 0$$
Substituting this into \eqref{eq:underbrace} shows that \eqref{eq:suffices} holds in this Case 2. 
\end{itemize}
 \end{proof}

\subsection{Decreasing stepsize} 

The previous lemma can be used with a constant stepsize $\eta[k]=\epsilon$ to prove a bound
on the mean squared error between $\theta[k]$ and $\theta^*$. However, 
the following lemma uses a decreasing stepsize to achieve a faster convergence.  
The proof uses an induction argument inspired by the analysis in 
\cite{Frank-Wolfe} for a different type of problem (deterministic Frank-Wolfe
methods for convex optimization). 
Due to the renewal optimization structure of the current problem, the 
stepsize used here must be sized carefully with respect to the $T_{min}$ parameter. 
This required care in choosing the stepsize  is 
analogous to the discussion on stepsize for a different class of systems in \cite{SGD-robust}:
Work in \cite{SGD-robust} shows how, for a class of systems with strongly convex properties, 
a fast $O(1/k)$ convergence rate can be degraded into a slow
$O(1/k^{1/5})$ rate if the stepsize parameter is not carefully sized according to a strong convexity parameter  (which may be difficult to know in practice).  That example is used to motivate alternative robust approaches for the systems studied there. 
The  lemma below applies to a different kind of system and does not require any type of strong convexity/concavity.  It uses a value $T_{min}$ for the stepsize selection, but $T_{min}$ is easy to know in practice. 
For example, if all frame sizes are at least 1 unit of time, we can use $T_{min}=1$.

\begin{lem} Under the algorithm \eqref{eq:alg1a}-\eqref{eq:theta-update} with any $\theta[0] \in [\theta_{min}, \theta_{max}]$ and with stepsizes 
$$\eta[k]=\frac{1}{(k+2)T_{min}} \quad \forall  k \in \{0, 1, 2, \ldots\}$$ 
we have 
\begin{equation} \label{eq:decreasing1} 
\expect{(\theta[k]-\theta^*)^2} \leq \frac{2b}{k T_{min}^2} \quad \forall k \in \{1, 2, 3, \ldots\}
\end{equation}
where the constant $b$ satisfies \eqref{eq:b}; the constant  $T_{min}$ satisfies \eqref{eq:moment1}-\eqref{eq:moment2}; the constant $\theta^*$ is defined by  \eqref{eq:theta-star}. 
\end{lem} 

\begin{proof} 
For simplicity define 
$$ z_k = \frac{1}{2}\expect{(\theta[k]-\theta^*)^2} \quad \forall k \in \{0, 1, 2, \ldots\}$$
It suffices to show $z_k\leq \frac{b}{kT_{min}^2}$ for all $k \in \{1, 2, 3, \ldots\}$. 
From \eqref{eq:iteration}  we have
\begin{equation} \label{eq:zk}
z_{k+1} \leq  (1-2T_{min}\eta[k])z_k + \eta[k]^2b \quad \forall k \in \{0, 1, 2, \ldots\}
\end{equation} 
Applying the above inequalty at $k=0$ and using $\eta[0]=1/(2T_{min})$ gives
$$ z_1 \leq \frac{b}{4T_{min}^2}$$
We now use induction with the base case $k=1$. 
Suppose that $z_k \leq b/(kT_{min}^2)$ for some $k \in \{1, 2, 3, \ldots\}$ (it holds for $k=1$ by the above inequality, since $b/(4T_{min}^2)\leq b/T_{min}^2$).  We show the same holds for $k+1$. We have from \eqref{eq:zk}: 
\begin{align*}
z_{k+1} &\leq (1-2T_{min}\eta[k])z_k + \eta[k]^2b \\
&\overset{(a)}{=} (\frac{k}{k+2})z_k + \frac{b}{(k+2)^2T_{min}^2}\\
&\overset{(b)}{\leq} (\frac{k}{k+2})\frac{b}{kT_{min}^2}+ \frac{b}{(k+2)^2T_{min}^2}\\
&= \frac{b(k+3)}{(k+2)^2T_{min}^2}\\
&\overset{(c)}{\leq} \frac{b}{(k+1)T_{min}^2}
\end{align*}
where (a) holds because $\eta[k]=\frac{1}{(k+2)T_{min}}$; (b) holds by the induction assumption 
$z_k\leq \frac{b}{kT_{min}^2}$; (c) holds because $(k+3)/(k+2)^2 \leq 1/(k+1)$ for all $k\geq 0$. 

\end{proof} 

\subsection{Online performance theorem} 

\begin{thm} \label{thm:1} (General performance) Under the algorithm \eqref{eq:alg1a}-\eqref{eq:theta-update} with $\theta[0] \in [\theta_{min}, \theta_{max}]$ and stepsizes $\eta[k]=\frac{1}{(k+2)T_{min}}$ for  $k \in \{0, 1, 2, \ldots\}$ we have 
\begin{align*}
\left|\theta^*-\frac{\sum_{k=0}^{K-1} \expect{R[k]}}{\sum_{k=0}^{K-1} \expect{T[k]}}\right| \leq \frac{\sqrt{2C_1}}{KT_{min}}\left[|\theta[0]-\theta^*| +\frac{-\sqrt{2b} + \sqrt{8b(K-1)}}{T_{min}}\right] \quad \forall K \in \{2, 3, 4, \ldots\} 
\end{align*}
where $C_1$ satisfies \eqref{eq:moment4} and 
$b$ satisfies \eqref{eq:b}. Hence, deviation from the optimal ratio $\theta^*$ decays like $O(1/\sqrt{K})$. 
\end{thm} 

\begin{proof} 
From \eqref{eq:abs-val} we have for all  integers $K\geq 2$: 
\begin{align*}
\left|\theta^*-\frac{\sum_{k=0}^{K-1} \expect{R[k]}}{\sum_{k=0}^{K-1}\expect{T[k]}}\right| &\leq  \frac{1}{KT_{min}}\sum_{k=0}^{K-1}\expect{|\theta[k]-\theta^*|\cdot|T[k]-t^*|}\\
&\overset{(a)}{\leq} \frac{1}{KT_{min}}\sum_{k=0}^{K-1}\sqrt{\expect{(\theta[k]-\theta^*)^2}\expect{(T[k]-t^*)^2}}\\
&\overset{(b)}{\leq} \frac{\sqrt{2C_1}}{KT_{min}}\left[\sqrt{\expect{(\theta[0]-\theta^*)^2}} + \sum_{k=1}^{K-1} \sqrt{\expect{(\theta[k]-\theta^*)^2}}\right]\\
&\overset{(c)}{\leq}  \frac{\sqrt{2C_1}}{KT_{min}}\left[\sqrt{\expect{(\theta[0]-\theta^*)^2}} +\sum_{k=1}^{K-1} \sqrt{\frac{2b}{kT_{min}^2}}\right]\\
&\overset{(d)}{\leq}  \frac{\sqrt{2C_1}}{KT_{min}}\left[\sqrt{\expect{(\theta[0]-\theta^*)^2}} + \frac{-\sqrt{2b} + \sqrt{8b(K-1)}}{T_{min}}\right]
\end{align*}
where (a) follows by the Cauchy-Schwarz inequality; 
(b) holds by \eqref{eq:moment4} and $(T[k]-t^*)^2 \leq T[k]^2 + (t^*)^2\leq T[k]^2 + C_1$; (c) holds by \eqref{eq:decreasing1}; (d) holds 
because $\sum_{k=1}^{K-1}\frac{1}{\sqrt{k}} \leq 1 + \int_1^{K-1} \frac{1}{\sqrt{t}}dt$. 
\end{proof}

\subsection{Strongly concave curvature} \label{section:strongly-concave} 

This section proves that the algorithm achieves a faster convergence rate in the special case when the set $\overline{\script{A}}$ has a strongly concave property. Specifically,  suppose 
the set $\overline{\script{A}}$ has a strongly concave upper boundary about the optimality 
point $(t^*,r^*)$, so that for some $c>0$ we have (see Fig. \ref{fig:example-strongly}): 
\begin{equation} \label{eq:curvature} 
r \leq r^* + \theta^*(t-t^*) - \frac{c}{2}(t-t^*)^2 \quad \forall (t,r) \in \overline{\script{A}}
\end{equation} 
\begin{thm} \label{thm:2} (Performance with strongly concave curvature) Assume $\overline{\script{A}}$ has the strongly concave curvature property specified in \eqref{eq:curvature}. Under the algorithm \eqref{eq:alg1a}-\eqref{eq:theta-update} with $\theta[0] \in [\theta_{min}, \theta_{max}]$ and stepsize $\eta[k]=\frac{1}{(k+2)T_{min}}$ for $k \in \{0, 1, 2, \ldots\}$ we have 
\begin{equation} \label{eq:achieve} 
\left|\theta^*-\frac{\sum_{k=0}^{K-1} \expect{R[k]}}{\sum_{k=0}^{K-1} \expect{T[k]}}\right| \leq \frac{2(\theta[0]-\theta^*)^2 + \frac{4b}{T_{min}^2}(1+\log(K-1))}{KcT_{min}} \quad \forall K \in \{2, 3, 4, \ldots\} 
\end{equation} 
and so deviation from the optimal ratio $\theta^*$ decays like $O(\log(K)/K)$. 
\end{thm} 
\begin{proof} 
For almost all $\theta[k]$ we have (by Lemma \ref{lem:tech-intro}): 
$$ (\expect{T[k]|\theta[k]}, \expect{R[k]|\theta[k]}) \in \overline{\script{A}}$$
And so by \eqref{eq:curvature} we have  with probability 1: 
\begin{equation} \label{eq:curvature-prob1} 
 \expect{R[k]|\theta[k]} \leq r^* + \theta^*(\expect{T[k]|\theta[k]}-t^*) - \frac{c}{2}(\expect{T[k]|\theta[k]}-t^*)^2   
 \end{equation} 
From \eqref{eq:R-theta-bound2} we have for all real numbers $\beta>0$ (with prob 1): 
\begin{align*}
\expect{R[k]-\theta^*T[k]|\theta[k]} &\geq (\theta[k]-\theta^*)(\expect{T[k]|\theta[k]} -t^*)\\
&= \left(\frac{1}{\beta}(\theta[k]-\theta^*)\right)\beta(\expect{T[k]|\theta[k]}-t^*)\\
&\overset{(a)}{\geq}-\frac{(\theta[k]-\theta^*)^2}{2\beta^2} - \frac{\beta^2(\expect{T[k]|\theta[k]}-t^*)^2}{2} \\
&\overset{(b)}{\geq} -\frac{(\theta[k]-\theta^*)^2}{2\beta^2} - \frac{\beta^2}{c}\left[r^* + \theta^*(\expect{T[k]|\theta[k]} - t^*) - \expect{R[k]|\theta[k]}\right]  \\
&\overset{(c)}{=} -\frac{(\theta[k]-\theta^*)^2}{2\beta^2} +\frac{\beta^2}{c}\expect{R[k]-\theta^*T[k]|\theta[k]}
\end{align*}
where (a)  uses the fact $ab \geq -\frac{a^2+b^2}{2}$ for all real numbers $a,b$; (b) uses \eqref{eq:curvature-prob1}; (c) uses $r^*-\theta^*t^*=0$. 
Choose $\beta>0$ so that $\beta^2/c = 1/2$.  Then
$$ \expect{R[k]-\theta^*T[k]|\theta[k]} \geq -\frac{(\theta[k]-\theta^*)^2}{\beta^2}= -2\frac{(\theta[k]-\theta^*)^2}{c}$$
where the final equality uses $\beta^2=c/2$. Taking expectations of the above and using the law of iterated expectations gives
$$\expect{R[k]-\theta^*T[k]} \geq  -2\frac{\expect{(\theta[k]-\theta^*)^2}}{c}$$
Fix integer $K\geq 2$. Summing the above inequality over $k \in \{0, \ldots, K-1\}$ gives 
\begin{align*}
\sum_{k=0}^{K-1} \expect{R[k]} - \theta^* \sum_{k=0}^{K-1}\expect{T[k]} &\geq -2\frac{\expect{(\theta[0]-\theta^*)^2}}{c} -2\sum_{k=1}^{K-1}\frac{\expect{(\theta[k]-\theta^*)^2}}{c}\\
&\overset{(a)}{\geq} -2\frac{\expect{(\theta[0]-\theta^*)^2}}{c} -2\sum_{k=1}^{K-1}\frac{2b}{kcT_{min}^2}\\
&\overset{(b)}{\geq}-2\frac{\expect{(\theta[0]-\theta^*)^2}}{c} - \frac{4b}{cT_{min}^2}(1+\log(K-1))
\end{align*}
where (a) holds by \eqref{eq:decreasing1}; (b) holds because $\sum_{k=1}^{K-1} 1/k \leq 1 + \int_1^{K-1} (1/t)dt$.  Rearranging terms gives 
\begin{align*}
\frac{\sum_{k=0}^{K-1} \expect{R[k]}}{\sum_{k=0}^{K-1} \expect{T[k]}} &\geq \theta^* - \frac{2\expect{(\theta[0]-\theta^*)^2} + \frac{4b}{T_{min}^2}(1+\log(K-1))}{c\sum_{k=0}^{K-1} \expect{T[k]}}\\
&\geq \theta^* - \frac{2\expect{(\theta[0]-\theta^*)^2} + \frac{4b}{T_{min}^2}(1+\log(K-1))}{KcT_{min}}
\end{align*}
where the final inequality holds because $\expect{T[k]}\geq T_{min}$ for all $k$.  On the other hand \eqref{eq:part1} implies that the ratio on the left-hand-side is less than or equal to $\theta^*$.  This proves the result. 
\end{proof} 

\subsection{Discussion} 

Theorem \ref{thm:1} shows the optimality gap decays like $O(1/\sqrt{K})$ for general systems. 
Theorem \ref{thm:2} shows that for systems with a strongly concave property, the 
optimality gap decays much faster according to $O(\log(K)/K)$.  This improved convergence speed does not require any changes in the algorithm itself.  Indeed, the algorithm does 
not need to know whether or not the system has the strongly concave property. If the system \emph{does} have the strongly concave property, the algorithm automatically yields faster convergence without having to know the strong concavity parameter $c$.

\section{Matching converse for strongly concave structure} \label{section:log-converse} 

This section constructs a particular system (with a strongly concave curvature) 
for which all causal algorithms that do not have a-priori knowledge of the probability distribution $F_S(s)$ have optimality gaps that decay no faster 
than $\Omega(\log(K)/K)$.  This matches the $O(\log(K)/K)$ achievability result of Theorem \ref{thm:2} and shows that this convergence rate is optimal over the class of systems with strongly concave curvature.  For the proof, 
we construct a nontrivial mapping from the sequential decision problem to a related estimation problem. 
This allows use of the Bernoulli estimation 
theorem of  \cite{hazan-kale-stochastic}. This mapping technique is conceptually similar to the method recently used to prove a converse result for a different class of systems with unit timeslots in \cite{neely-NUM-converse-infocom2020}.

\subsection{System} \label{section:system-converse} 

Suppose the task type process is an 
i.i.d. Bernoulli process $\{S[k]\}_{t=0}^{\infty}$ with 
\begin{equation} \label{eq:example-S}
P[S[k]=1] = q \quad ; \quad P[S[k]=0]=1-q
\end{equation} 
where $q$ is an unknown probability.  For technical reasons, we assume 
throughout that $q \in [1/4,3/4]$. 
Every frame $k\in \{0, 1, 2, \ldots\}$ the controller observes $S[k]$ and then 
chooses $(T[k], R[k]) \in \script{D}(S[k])$, where 
\begin{equation} \label{eq:example-D} 
 \script{D}(S[k]) =  \left\{ \begin{array}{ll}
(1,1) &\mbox{ , if $S[k]=0$} \\
\{(x, 2-(2-x)^2) \in \mathbb{R}^2 : x \in [1,2] \}& \mbox{ , if $S[k]=1$} 
\end{array}
\right.
\end{equation} 
The decision structure \eqref{eq:example-D} defines a system with \emph{inflexible tasks} (type 0) and \emph{flexible tasks} (type 1).
Indeed, if $S[k]=0$ then the controller must choose  $(T[k], R[k]) = (1,1)$. 
However, if  $S[k]=1$ then the controller can choose $(T[k],R[k])$ 
as any point on the curve $(x, 2 - (2-x)^2)$ for $x \in [1,2]$. 
A higher reward is obtained for larger values of $x$, but with diminishing
returns (see Fig. \ref{fig:example-strongly}).  This particular curve  is chosen as 
a representative example with \emph{strongly concave curvature}.\footnote{The strongly concave decision curve gives rise to a set $\overline{\script{A}}$ with a strongly concave upper boundary. The $\Omega(\log(K)/K)$ converse
result in  this section can be extended to more general systems for which $\overline{\script{A}}$ has a strongly concave upper boundary.} 

\begin{figure}[htbp]
   \centering
   \includegraphics[width=4in]{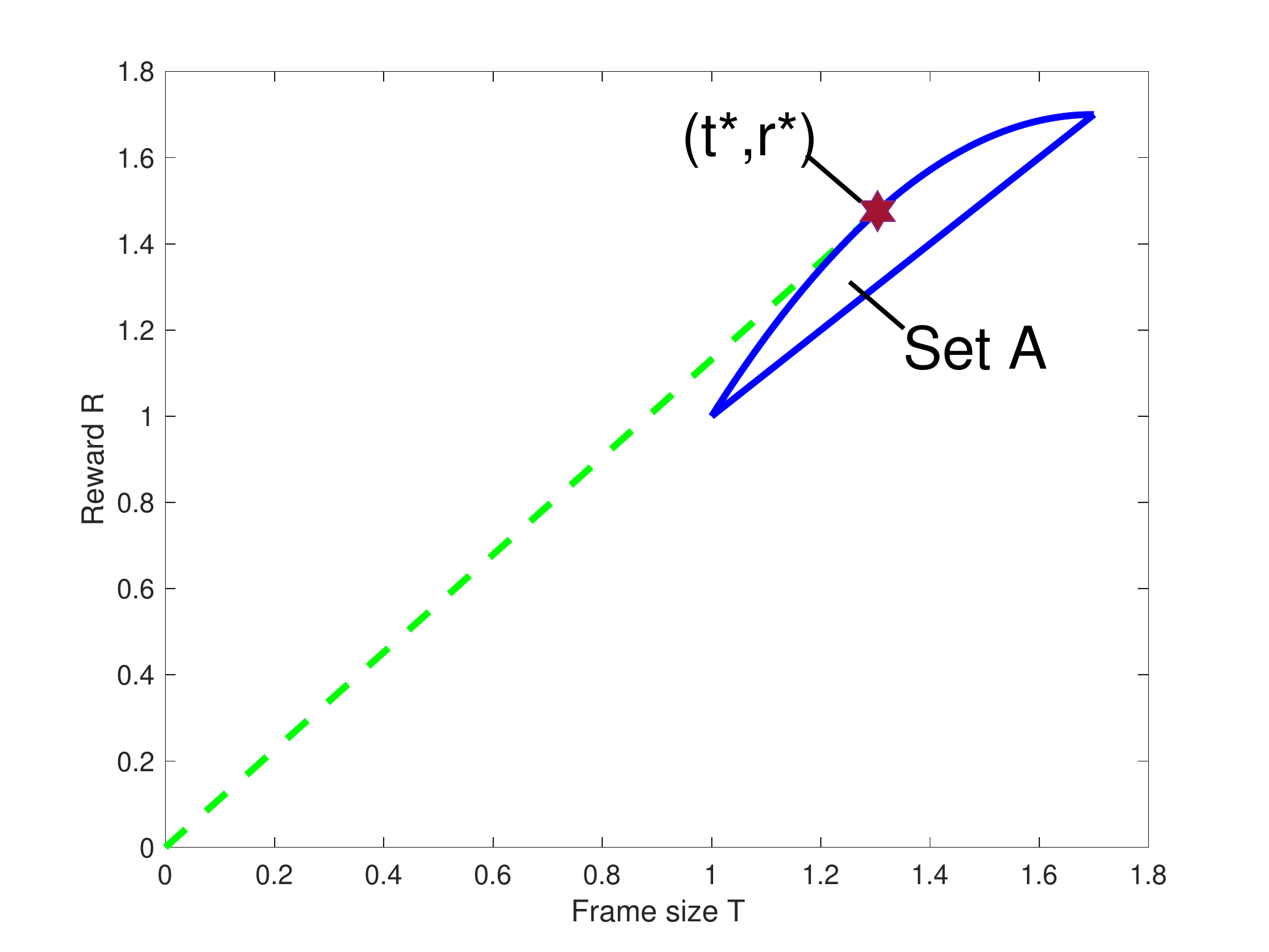} 
   \caption{The set $\script{A}$ with a strongly concave upper boundary 
   for the example of Section \ref{section:system-converse} with $q=0.7$.}
   \label{fig:example-strongly}
\end{figure}

\begin{lem} \label{eq:lem-W}  Fix $q \in [1/4,3/4]$.  Let $\script{A}$ be the set of one-shot expectations
for the example problem 
\eqref{eq:example-S}-\eqref{eq:example-D}. 
Define the sets
\begin{align*}
\script{Z} &= \left\{\left(1-q+qx, 1+q-q(2-x)^2\right) \in \mathbb{R}^2 :  x \in [1,2]\right\}\\
\script{W} &= \left\{(t,r) \in \mathbb{R}^2: t \in [1,1+q], t \leq r \leq (2t-1) - \frac{1}{q}(t-1)^2 \right\}
\end{align*}
Then 
$$ \script{A} = Conv(\script{Z}) = \script{W}$$
where $Conv(\script{Z})$ denotes the convex hull of the set $\script{Z}$. In particular, $\script{A}$ is closed and so $\overline{\script{A}} = \script{A}$. 
\end{lem} 

\begin{proof} 
(Part 1): We show $Conv(\script{Z})\subseteq \script{A}$. 
Fix $z \in \script{Z}$.  Then there is an $x \in [1,2]$ such that
$$ z = (1-q + qx, 1+q-q(2-x)^2) $$
The policy that chooses $(T[0],R[0])=(1,1)$ when $S[0]=0$ and $(T[0],R[0]) = (x, 2-(2-x)^2)$ when 
$S[0]=1$ yields: 
\begin{align*}
\expect{(T[0], R[0])} &= (1-q)\cdot(1,1) + q\cdot(x, 2-(2-x)^2) \\
&=(1-q + qx, 1+q-q(2-x)^2)\\
&=z
\end{align*}
and so the vector $z$ is achievable as an expectation on slot $0$, meaning $z \in \script{A}$. This holds for all $z \in \script{Z}$ and so 
$$\script{Z} \subseteq \script{A} \implies Conv(\script{Z}) \subseteq Conv(\script{A}) = \script{A}$$
where the final equality holds because the set $\script{A}$ is convex and so it is equal to its convex hull. 

(Part 2): We show $\script{A} \subseteq \script{W}$. 
Fix $a \in \script{A}$. By definition of  $\script{A}$ as the region of one-shot expectations, 
there is a policy that yields decisions $(T[0], R[0]) \in \script{D}(S[0])$ such that: 
$$ (\expect{T[0]}, \expect{R[0]}) = a$$
To prove $a \in \script{W}$ it suffices to show that 
\begin{align} 
&1\leq \expect{T[0]} \leq 1+q \label{eq:suffice-1} \\
& \expect{T[0]} \leq \expect{R[0]} \leq (2\expect{T[0]}-1) -\frac{1}{q}(\expect{T[0]}-1)^2 \label{eq:suffice-2}
\end{align}
Regardless of the value of $S[0]$ we have from the structure of $\script{D}(S[0])$ in \eqref{eq:example-D}: 
$$ 1 \leq T[0]\leq 1+1_{\{S[0]=1\}}$$
where $1_{\{S[0]=1\}}$ is an indicator function that is 1 if $S[0]=1$ and $0$ else. 
Taking expectations of the above inequality yields \eqref{eq:suffice-1}.  Again regardless of $S[0]$, 
we have 
\begin{equation} \label{eq:sample-path}
T[0] \leq R[0] = (2T[0]-1) - (T[0]-1)^2
\end{equation} 
Indeed, if $S[0]=0$ then $T[0]=R[0]=1$ and so \eqref{eq:sample-path} holds, 
while if $S[0]=1$ then \eqref{eq:sample-path} again holds because
by \eqref{eq:example-D} we have 
\begin{align*}
R[0] &=  2-(2-T[0])^2 \\
&\geq T[0]
\end{align*}
where the final inequality holds because $T[0] \in [1,2]$. Taking expectations of \eqref{eq:sample-path} gives
\begin{align*}
\expect{T[0]} &\leq \expect{R[0]} \\
&= (2\expect{T[0]}-1) - \expect{(T[0]-1)^2}\\
&\overset{(a)}{=}   (2\expect{T[0]}-1) - q\expect{(T[0]-1)^2|S[0]=1}\\
&\overset{(b)}{\leq} (2\expect{T[0]}-1) - q(\expect{T[0]-1|S[0]=1})^2 \\
&= (2\expect{T[0]}-1) - \frac{1}{q}(q\expect{T[0]-1|S[0]=1})^2\\
&\overset{(c)}{=} (2\expect{T[0]}-1) - \frac{1}{q}(\expect{T[0] - 1})^2
\end{align*}
where (a) holds because $T[0]=1$ whenever $S[0]=0$; (b) holds by Jensen's inequality applied to the conditional expectation; 
(c) holds because $\expect{T[0]-1} = q\expect{T[0]-1|S[0]=1}$ (since 
$T[0]-1=0$ if $S[0]=0$).  This establishes \eqref{eq:suffice-2} and thus proves $\script{A} \subseteq \script{W}$. 

(Part 3): We show $\script{W} \subseteq Conv(\script{Z})$. 
This can be shown by noting $\script{W}$ is the convex hull of the following set of points
$$ \{(t, (2t-1) - \frac{1}{q}(t-1)^2) : t \in [1,1+q]\}$$
and these points coincide with the set $\script{Z}$. 
\end{proof} 

\begin{lem} \label{lem:theta-star-log} The point $(t^*, r^*)\in \script{A}$ that maximizes $r/t$ over all points $(t,r) \in \script{A}$ 
is achieved on the upper boundary of $\script{A}$ (see Fig. \ref{fig:example-strongly}) and satisfies: 
\begin{align}
t^* &= \sqrt{1+q} \label{eq:t-star}  \\
r^* &=  \frac{2(q+1)\left(-1+\sqrt{1+q}\right)}{q} \label{eq:r-star} \\
\theta^* &= \frac{r^*}{t^*} = 2 - \frac{2}{q}\left(-1+\sqrt{1+q}\right) \label{eq:example-theta} 
\end{align}
Further, by strong concavity of the upper boundary of set $\script{A}$ we have
\begin{equation} \label{eq:bound-t-r}
 r \leq \theta^* t - \frac{1}{q}(t-t^*)^2 \quad \forall (t,r) \in \script{A}
 \end{equation} 
\end{lem} 

\begin{proof}  By the previous lemma, the upper boundary of $\script{A}$ consists of points $(t,r)$ such that 
\begin{equation} \label{eq:parameterize1}
 r = (2t-1) - \frac{1}{q}(t-1)^2 \quad, t \in [1,1+q]
 \end{equation}
It is not difficult to show the maximizer of $r/t$ over all $(t,r) \in \script{A}$ is on this upper boundary and is given by the point $(t^*,r^*)$ specified in \eqref{eq:t-star}-\eqref{eq:example-theta} (details omitted for brevity). The final inequality \eqref{eq:bound-t-r} hold because the upper boundary curve \eqref{eq:parameterize1} can be reparameterized by
$$ r = \theta^* t - \frac{1}{q}(t-t^*)^2, t \in [1, 1+q]$$
which is a quadratic that matches the squared term and constant term in the quadratic 
\eqref{eq:parameterize1} and also matches at the point $(t^*,r^*)$.  
\end{proof} 

\subsection{Causal and measurable algorithms} \label{section:causal}

Our converse result considers algorithms that choose $(T[k], R[k]) \in \script{D}(S[k])$ every 
frame $k$ in a way that is \emph{causal} (so the algorithm has no knowledge of the future) 
and \emph{probabilistically measurable} (so probability distributions and expectations are well
defined).     
For each $k \in \{0, 1, 2, \ldots\}$, define $H[k]$ as the system history up to but not including 
frame $k$: 
$$ H[k] = (S[0], S[1], \ldots, S[k-1]) \quad \forall k \in \{1, 2, 3, \ldots\}$$
where $H[0]$ is formally defined as NULL (since there is no history before frame $0$). 
On frame $k$, a general algorithm has no choice if it observes $S[k]=0$ (since it must use $(T[k],R[k])=(1,1)$ in that case) but must select $(T[k], R[k]) \in \script{D}(1)$ if it observes $S[1]$. A general causal and measurable algorithm must make this (potentially randomized) 
decision based only on the history $H[k]$ observed on frames $\{0, 1, \ldots, k-1\}$. 

It suffices to assume that on frame $k$, the algorithm makes decisions as a deterministic function 
of $(H[k], U)$ where $U$ is uniformly distributed in $[0,1)$ and is independent of $\{S[k]\}_{k=0}^{\infty}$.  Let $\{f_k\}_{k=0}^{\infty}$ 
be a sequence of  functions of the following form: 
\begin{align*}
 f_k:\{0,1\}^{k} \times [0,1] \rightarrow \script{D}(1) \quad \forall k \in \{1, 2, 3, \ldots\} 
 \end{align*}
 where $\{0,1\}^{0}$ is formally defined as NULL so that $f_0(u)$ depends only on 
 $u \in [0,1]$. These functions $\{f_k\}_{k=0}^{\infty}$ define control decisions of a general causal and measurable algorithm on each frame $k \in \{0, 1, 2, \ldots\}$: 
 \begin{align}
 (T[k], R[k]) = \left\{ \begin{array}{ll}
(1,1) &\mbox{ , if $S[k]=0$} \\
f_k(H[k], U) & \mbox{ , if $S[k]=1$} 
\end{array}
\right.   \label{eq:causal-TR} 
 \end{align}
 The functions $\{f_k\}_{k=0}^{\infty}$ are assumed to be \emph{probabilistically measurable} on the  
 sigma algebra $\sigma(U; \{S[k]\}_{k=0}^{\infty})$.  Since $\script{D}(1)$ is a bounded set, the corresponding
 expectations $(\expect{T[k]}, \expect{R[k]})$ are well defined and finite, as are the conditional expectations
 given $(H[k], U)$. 
 We say that decisions $(T[k], R[k]) \in \script{D}(S[k])$ for all $k \in \{0, 1, 2, \ldots\}$ are \emph{causal and measurable} if they come from a sequence of deterministic functions $\{f_k\}_{k=0}^{\infty}$ with the above structure.  
 
 This algorithm structure is not restrictive: From the single random variable  $U$, 
 we can construct an infinite sequence of i.i.d. uniformly 
 distributed random variables $\{U_i\}_{i=1}^{\infty}$, where each $U_i$ is a deterministic function of $U$.\footnote{This can be done by writing $U=\sum_{i=1}^{\infty} X_i 10^{-i}$ in the unique base-10 expansion where $X_i \in \{0, \ldots, 9\}$ is the $i$th digit of the expansion and  $\{X_i\}_{i=1}^{\infty}$ does not have an infinite tail of 9s, 
 and defining $U_n$ for each $n \in \{1, 2, 3 ,\ldots\}$ by  $U_n = \sum_{i=1}^{\infty} X_{g(n,i)}10^{-i}$ 
where $g:\mathbb{N}^2\rightarrow\mathbb{N}$ is any bijection.} 
This allows the controller to 
make as many calls to a random number generator as needed 
(assuming an at-most countably infinite number of such calls). 

\begin{lem} \label{lem:in-A}
Fix $q \in [1/4, 3/4]$ and consider any causal and measurable decisions $(T[k], R[k]) \in \script{D}(S[k])$ 
for $k \in \{0, 1, 2, \ldots\}$.  For each  $k \in \{0, 1, 2, \ldots\}$  we have: 
\begin{equation} \label{eq:in-A1} 
\expect{(T[k], R[k]) | H[k]} \in \script{A}
\end{equation} 
Furthermore,  
\begin{equation} \label{eq:in-A2}
\expect{R[k]|H[k]} \leq \theta^* \expect{T[k]|H[k]} - \frac{1}{q}(\expect{T[k]|H[k]} - t^*)^2
\end{equation} 
where $t^*=\sqrt{1+q}$ and $\theta^*$ satisfies \eqref{eq:example-theta}. 
\end{lem} 

\begin{proof} 
The fact \eqref{eq:in-A1} is similar to Lemma \ref{lem:tech-intro} and is formally
proven by noting that $\overline{\script{A}}=\script{A}$ and 
$S[k]$ is independent of $(U, H[k])$, and then  
applying the general theorem in the Appendix (details omitted for brevity). 
Substituting \eqref{eq:in-A1} into  \eqref{eq:bound-t-r} proves \eqref{eq:in-A2}.
 \end{proof}

\subsection{The Bernoulli estimation theorem from \cite{hazan-kale-stochastic}} 

Let $\{g_k\}_{k=1}^{\infty}$ be an infinite sequence of deterministic 
functions of the form:
\begin{align*}
g_k :\{0,1\}^{k}   \rightarrow [0,1] \quad \forall k \in \{1, 2, 3, \ldots\} 
\end{align*}
and  for each $k \in \{1, 2, 3, \ldots\}$ the 
function $g_k(s_0, \ldots, s_{k-1})$ maps a binary-valued sequence $(s_0, \ldots, s_{k-1})$ 
to a real number in the interval $[0,1]$.  Let $\{S[k]\}_{k=0}^{\infty}$ be an i.i.d. sequence of Bernoulli random variables with
parameter $q = P[S[k]=1]$.  Assume that $q$ is an unknown parameter in the interval $[1/4,3/4]$. 
The functions $g_k$ shall
be called \emph{estimation functions} because they can map the first $k$ observations of the Bernoulli random variables 
to a (deterministic)  estimate $G[k]$ of the $q$ parameter via: 
\begin{align}
G[k] = g_k(S[0], S[1], \ldots, S[k-1]) \quad \forall k \in \{1, 2, 3, \ldots\}  \label{eq:Gk} 
\end{align}
The following theorem from \cite{hazan-kale-stochastic} 
provides a lower bound on the mean square error for any sequence of estimation functions.\footnote{The
result in \cite{hazan-kale-stochastic} also applies to randomized estimation functions.  This paper shall only need
the result for deterministic estimation functions.} 

\begin{thm} \label{thm:hazan} (Bernoulli Estimation  \cite{hazan-kale-stochastic})
For any sequence of estimation functions $\{g_k\}_{k=0}^{\infty}$ as defined above, 
there exists a probability $q \in [1/4, 3/4]$ such that if $\{S[k]\}_{k=0}^{\infty}$ is an i.i.d. Bernoulli sequence with parameter $q$,  then for all positive integers $K$ we have 
\begin{equation} \label{eq:bound} 
 \sum_{k=1}^{K} \expect{(q-G[k])^2} \geq \Omega(\log(K))
 \end{equation} 
where the random variables $G[k]$ are defined in \eqref{eq:Gk} for each $k \in \{1, 2, \ldots\}$. 
\end{thm}

\subsection{Completing the converse}

\begin{lem} Define $\phi:[1/4,3/4] \rightarrow \mathbb{R}$ by 
$$ \phi(q) = \frac{-1+\sqrt{1+q}}{q}$$
Then 

a)  $\phi$ is continuous and strictly decreasing with minimim and maximum values
\begin{align*}
\phi_{min} &= \phi(3/4) \approx 0.430501  \\
\phi_{max} &= \phi(1/4) \approx 0.472136
\end{align*}

b) There is a continuous  
inverse function $\phi^{-1}:[\phi_{min}, \phi_{max}]\rightarrow [1/4, 3/4]$ 
that satisfies 
$$ \phi(\phi^{-1}(y)) = y \quad \forall y \in [\phi_{min}, \phi_{max}]$$

c) The derivative of $\phi$ exists for all $q \in (1/4, 3/4)$ and 
$$ |\phi'(q)| \geq \beta \quad \forall q \in (1/4, 3/4)$$
where $\beta = |\phi'(3/4)| \approx 0.0700485$.

d) We have 
\begin{equation} \label{eq:last-part} 
|\phi(q) - \phi(u)| \geq \beta |q-u| \quad \forall u,q \in [1/4, 3/4]
\end{equation} 
\end{lem}

\begin{proof} 
Parts (a) and (c) follow by basic analysis on the function $\phi$. Part (b) follows immediately 
from (a).  To prove (d), without loss of generality assume  $1/4 \leq u < q \leq 3/4$. 
By the mean value theorem, there is a  
point $x \in (u,q)$ such that 
 $$ \frac{\phi(q)-\phi(u)}{q-u} = \phi'(x) $$
 and so 
 $$ \left|\frac{\phi(q)-\phi(u)}{q-u}\right| = |\phi'(x)|\geq \beta $$
\end{proof}

\begin{thm} 
Let $\{f_k\}_{k=0}^{\infty}$ be any sequence of decision functions that define a causal and measurable algorithm according to \eqref{eq:causal-TR}. 
There is a parameter $q \in [1/4, 3/4]$ such that using these decision functions with an i.i.d. Bernoulli-$q$ process
$\{S[k]\}_{k=0}^{\infty}$ and with an independent uniform random variable $U$, the resulting 
(causal and measurable) decisions $(T[k], R[k]) \in \script{D}(S[k])$ satisfy 
$$ \frac{\sum_{k=0}^{K-1} \expect{R[k]}}{\sum_{k=0}^{K-1} \expect{T[k]}} \leq \theta^* - \Omega\left(\frac{\log(K)}{K}\right)$$
where $\theta^*$ is the optimal ratio in \eqref{eq:example-theta}. 
In particular, no algorithm can have  error that decays faster than $\Omega(\log(K)/K)$. 
\end{thm} 

\begin{proof} 
Define for each $k \in \{0, 1, 2, \ldots\}$ and each $h_k \in \{0,1\}^{k}$: 
$$z_k(h_k) = \expect{T[k]|S[k]=1, H[k]=h_k}$$
Define $g_k:\{0,1\}^{k}\rightarrow [0,1]$
 by 
 $$ g_k(h_k) = \phi^{-1}\left([z_k(h_k)-1]_{\phi_{min}}^{\phi_{max}}\right)$$
 Since $\phi^{-1}:[\phi_{min}, \phi_{max}]\rightarrow [1/4, 3/4]$, these $\{g_k\}$ functions can be viewed as estimation functions. 
 Define 
 \begin{equation} \label{eq:G1} 
  G[k] = g_k(H[k]) \quad \forall k \in \{1, 2, 3, \ldots\} 
  \end{equation} 
 and so 
 \begin{equation} \label{eq:G2} 
 G[k] = \phi^{-1}\left([  z_k(H_k)-1  ]_{\phi_{min}}^{\phi_{max}} \right)
 \end{equation} 
 Observe that for $k \in \{1, 2, 3, \ldots\}$, 
 the $g_k$ functions can be viewed as estimation functions for the parameter $q$, with $G[k]$ the corresponding estimator based on the past $k$ observations, because they have the general structure specified by the Bernoulli estimation theorem (Theorem \ref{thm:hazan}). Indeed $G[k]$ maps $H[k]=(S[0], \ldots, S[k-1])$ to the unit interval $[0,1]$, and this map is measurable because it is a composition with the measurable $z_k(h)$ function, the continuous projection to the interval $[\phi_{min}, \phi_{max}]$, and the continuous inverse function $\phi^{-1}$.  Hence, if we can expression the ratio of expectations in question as a sum of the mean squared error between $G[k]$ and $q$, we can apply Theorem \ref{thm:hazan}.
 
Fix $k \in \{0, 1, 2, \ldots\}$. We have for each $h_k \in \{0,1\}^{k}$: 
\begin{align*}
\expect{R[k]|H[k]} &\overset{(a)}{\leq} \theta^*\expect{T[k]|H[k]} - \frac{1}{q}\left(\expect{T[k]|H[k]} - t^*\right)^2  \\
&\overset{(b)}= \theta^*\expect{T[k]|H[k]} - \frac{1}{q} \left((1-q) + qz_k(H[k]) - t^*\right)^2\\
&= \theta^*\expect{T[k]|H[k]} - q\left(z_k(H[k]) - 1 - \frac{-1+t^*}{q}\right)^2\\
&\overset{(c)}{=} \theta^*\expect{T[k]|H[k]} - q\left(z_k(H[k]) -1 - \phi(q)\right)^2\\
&\overset{(d)}{\leq}  \theta^* \expect{T[k]|H[k]} - q\left([z_k(H[k])-1]_{\phi_{min}}^{\phi_{max}}- \phi(q)\right)^2\\
&= \theta^* \expect{T[k]|H[k]} - q\left(\phi(\phi^{-1}([z_k(H[k])-1]_{\phi_{min}}^{\phi_{max}}))- \phi(q)\right)^2\\
&\overset{(e)}{=}\theta^* \expect{T[k]|H[k]} - q\left(\phi(G[k]) - \phi(q)\right)^2\\
&\overset{(f)}{\leq} \theta^* \expect{T[k]|H[k]} - q\beta^2(G[k]-q)^2\\
&\overset{(g)}{\leq} \theta^*\expect{T[k]|H[k]} - (1/4) \beta^2(G[k]-q)^2
\end{align*}
where (a) holds by \eqref{eq:in-A2}; (b) holds by  \eqref{eq:causal-TR}; (c) holds by definition of $\phi$ and because $t^*=\sqrt{1+q}$ (recall \eqref{eq:t-star}); (d) holds because the distance between $(z[k]-1)$ and $\phi(q)$ is greater than or equal to the distance between their projections onto the interval $[\phi_{min}, \phi_{max}]$ (and the fact that $\phi(q) \in [\phi_{min}, \phi_{max}]$); (e) holds by 
\eqref{eq:G2}; (f) holds by \eqref{eq:last-part}; (g) holds because $q \geq 1/4$.

Taking expectations of both sides with respect to the random $H[k]$ gives 
$$ \expect{R[k]} \leq \theta^*\expect{T[k]} - \frac{\beta^2}{4} \expect{(G[k]-q)^2} $$
Summing gives 
\begin{align*} 
\sum_{k=0}^{K-1} \expect{R[k]} &\leq \theta^*\sum_{k=0}^{K-1} \expect{T[k]}  - \frac{\beta^2}{4} \sum_{k=0}^{K-1} \expect{(G[k]-q)^2} \\
&\leq \theta^*\sum_{k=0}^{K-1} \expect{T[k]}   - \frac{\beta^2}{4} \sum_{k=1}^{K-1} \expect{(G[k]-q)^2} 
\end{align*}
Dividing by the expression $\sum_{k=0}^{K-1} \expect{T[k]}$ and noting that this expression is less than or equal to $2K$ (since $T[k]\leq 2$ always, see \eqref{eq:example-D}) yields
\begin{align*}
\frac{\sum_{k=0}^{K-1} \expect{R[k]}}{\sum_{k=0}^{K-1} \expect{T[k]}} &\leq \theta^* - \frac{\beta^2}{8K}\sum_{k=1}^{K-1} \expect{(G[k]-q)^2}\\
&\leq \theta^* - \Omega(\log(K)/K)
\end{align*}
where the final inequality holds by application of the Bernoulli estimation bound
\eqref{eq:bound}.
\end{proof} 

\section{Matching square root converse} \label{section:square-root} 

This section presents a square-root converse result for systems without the strongly concave structure.  We consider the example system of Section \ref{section:example} with 
task type process $\{S[k]\}_{k=0}^{\infty}$ that is i.i.d. Bernoulli with $P[S[k]=1]=p$, where $p$ is an unknown parameter. The nature of this converse is different from the one in the previous section and we shall consider only two possible values of $p$:  Fix $\epsilon$ such that $0<\epsilon\leq 1/4$. Consider the two possible hypotheses: 
\begin{itemize} 
\item Hypothesis $H_{(1/2-\epsilon)}$:  $p=1/2-\epsilon$.  Under this hypothesis it can be shown that: 
\begin{equation} \label{eq:H-neg}
\theta^*_{(1/2-\epsilon)} = 2 + 2\epsilon 
\end{equation} 
\item Hypothesis $H_{(1/2+\epsilon)}$: $p=1/2 + \epsilon$. Under this hypothesis it can be shown  that: 
\begin{equation} \label{eq:H-pos}
 \theta^*_{(1/2+\epsilon)} = \frac{6}{3 + 2\epsilon}
 \end{equation} 
\end{itemize} 
The structure of considering only two possible hypotheses that are difficult to discern is similar in spirit to the converse result of \cite{bubeck2012regret} for multi-armed bandit systems, where a square root law is also developed.  However, the square root arises for a different reason here. Indeed, the system treated in this paper has a different structure and requires a different proof. 

Fix $U$ uniform over $[0,1)$ and assume $U$ is independent of $\{S[k]\}_{k=0}^{\infty}$.  Consider a general causal algorithm that has no knowledge of $p$ and that makes decisions as follows: On each frame $k$ it chooses a conditional probability $\beta[k]$, being the conditional probability of choosing high quality given that $S[k]=1$, as some deterministic and measurable function $\hat{\beta}_k(\cdot)$ of $U$ and $S[0], \ldots, S[k-1]$: 
\begin{align}
\beta[0] &= \hat{\beta}_0(U) \label{eq:beta1} \\
\beta[k] &= \hat{\beta}_k(U, S[0], \ldots, S[k-1]) \quad \forall k \in \{1, 2, 3, \ldots\} \label{eq:beta2} 
\end{align}
For each frame $k$, given $\beta[k]$, we have 
$$ (T[k], R[k]) = \left\{\begin{array}{cc} 
(1,3) & \mbox{with prob $1-p$} \\
(2,3) & \mbox{with prob $p\beta[k]$} \\
(1, 1) & \mbox{with prob $p(1-\beta[k])$}\end{array}\right.$$
and so 
\begin{align}
\expect{R[k]|\beta[k]} &= (1-p)(3) + p(1+2\beta[k]) \nonumber \\
\expect{T[k]|\beta[k]} &= (1-p)(1) + p(1+\beta[k]) \nonumber \\
\expect{R[k] - \theta^* T[k]|\beta[k]} &= (1-p)(3) + p(1+2\beta[k]) - \theta^*(1-p + p(1+\beta[k]))\label{eq:R-theta-square-root} 
\end{align}

\begin{thm} \label{thm:square-root-converse}  Fix $\delta$ such that $0<\delta\leq 1/256$. 
Consider any  general causal algorithm of the type \eqref{eq:beta1}-\eqref{eq:beta2}.  
Fix $K \in \{1, \ldots, \lfloor \frac{3}{2^{19}\delta^2}\rfloor\}$. If for the case $p=\frac{1}{2}-64\delta$ the algorithm satisfies: 
$$ \frac{\sum_{k=0}^{K} \expect{R[k]}}{\sum_{k=0}^K\expect{T[k]}} > \theta^* - \delta$$
then the algorithm fails to achieve this for the case $p=\frac{1}{2}+64\delta$.
\end{thm}

The proof is developed in the next subsections by defining $\epsilon = 64\delta$ and by fixing a particular causal algorithm of the type described in this section.

\subsection{Case $H_{(1/2-\epsilon)}$} 

Fix $\epsilon = 64\delta$.  Fix a particular causal algorithm of the type \eqref{eq:beta1}-\eqref{eq:beta2}. Substituting \eqref{eq:H-neg} into \eqref{eq:R-theta-square-root} and doing the basic but tedious arithmetic gives
\begin{align*}
\mathbb{E}_{(1/2-\epsilon)}\left[R[k]-\theta^*_{(1/2-\epsilon)}T[k]|\beta[k]\right] &=  -\beta \epsilon + 2\beta \epsilon^2\\
&\overset{(a)}{\leq} -\frac{\beta[k]}{2} \epsilon\\
&\overset{(b)}{\leq}  -\frac{\epsilon}{4}1_{\{\beta[k]>1/2\}}
\end{align*}
where where $\mathbb{E}_{(1/2-\epsilon)}\left[\cdot\right]$ denotes an expectation under the assumption that the true Bernoulli parameter is $p=1/2-\epsilon$; 
 (a) holds because $0<\epsilon \leq 1/4$; (b) uses the indicator function $1_{\{\beta[k]>1/2\}}$ that is $1$ if $\beta[k]>1/2$ and $0$ else. Taking expectations and using the law of iterated expectations gives
\begin{equation} \label{eq:H-neg-case} 
\mathbb{E}_{(1/2-\epsilon)}\left[R[k] - \theta^*_{(1/2-\epsilon)}T[k]\right] \leq -\frac{\epsilon}{4}P[\beta[k]> 1/2]
\end{equation}

\subsection{Case $H_{(1/2+\epsilon)}$} 
Fix $\epsilon = 64\delta$.  Substituting \eqref{eq:H-pos} into \eqref{eq:R-theta-square-root} 
\begin{align*}
\mathbb{E}_{(1/2+\epsilon)}\left[R[k]-\theta^*_{(1/2+\epsilon)}T[k]|\beta[k]\right] &=  -(1-\beta[k]) \frac{2\epsilon + 4\epsilon^2}{3+2\epsilon}\\
&\overset{(a)}{\leq} -(1-\beta[k])\frac{\epsilon}{2}\\
&\overset{(b)}{\leq}  -\frac{\epsilon}{4}1_{\{\beta[k]\leq1/2\}}
\end{align*}
where $\mathbb{E}_{(1/2+\epsilon)}\left[\cdot\right]$ denotes an expectation under the assumption that the true Bernoulli parameter is $p=1/2+\epsilon$; (a) holds because $(2+ 4\epsilon)/(3+2\epsilon) \geq 1/2$ whenever $0<\epsilon<1/4$; (b) uses the indicator function $1_{\{\beta[k]\leq1/2\}}$ that is $1$ if $\beta[k]\leq1/2$ and $0$ else.  Taking expectations gives
\begin{equation} \label{eq:H-pos-case} 
\mathbb{E}_{(1/2+\epsilon)}\left[R[k] - \theta^*_{(1/2+\epsilon)}T[k]\right] \leq -\frac{\epsilon}{4}P[\beta[k]\leq 1/2]
\end{equation}

\subsection{Bernoulli estimation for mean absolute error} 

For each $k \in \{1, 2, 3, \ldots\}$ let $\hat{A}_k$ be Borel-measurable functions of the type: 
$$\hat{A}_k:[0, 1) \times \{0,1\}^k \rightarrow [0,1]$$
and define 
$$ A[k] = \hat{A}_k(U, S[0], S[1], \ldots, S[k-1])$$
The sequence of functions $\{A_k(\cdot)\}_{k=1}^{\infty}$ shall be called \emph{estimation functions} because they represent any way of estimating the Bernoulli parameter $p$ based only on $U$ and the first $k$ observations $S[0], \ldots, S[k-1]$. 

Let $A[k]_p$ be the resulting (random) estimation of the Bernoulli parameter given that the true parameter is $p$. Let $\mathbb{E}_p\left[|A[k]_p - p|\right]$ denote the expected mean absolute error given the true parameter is $p$. The following Bernoulli estimation lemma for mean absolute error 
is from \cite{neely-converse-NUM-arxiv} and is a modified version of a lemma for mean squared error developed in \cite{hazan-kale-stochastic}: 

\begin{lem} \label{lem:abs-error} For any sequence of Borel-measurable functions $\hat{A}_k(\cdot)$ of the structure defined above, we have 
$$ \mathbb{E}_p\left[|A[k]_p - p|\right] + \mathbb{E}_q\left[|A[k]_q - q|\right] \geq \frac{|p-q|}{4} \quad \mbox{ whenever $|p-q|\leq \frac{\sqrt{3}}{4\sqrt{2k}}$}$$
\end{lem} 

\begin{proof} 
See Lemma 6 in \cite{neely-converse-NUM-arxiv}.
\end{proof} 

To use the above lemma, define the following estimator 
functions for each $k \in \{1, 2, 3, \ldots\}$: 
$$ \hat{A}_k(u, s_0, \ldots, s_{k-1}) = \left\{\begin{array}{cc}
1/2-\epsilon & \mbox{ if $\hat{\beta}_k(u, s_0, \ldots, s_{k-1})\leq 1/2$} \\
1/2+\epsilon & \mbox{ if $\hat{\beta}_k(u, s_0, \ldots, s_{k-1})>1/2$} 
\end{array}\right.$$
where $(u, s_0, \ldots, s_{k-1}) \in [0, 1) \times \{0,1\}^k$.  Then 
\begin{align}
\mathbb{E}_{(1/2-\epsilon)}\left[|A[k]_{1/2-\epsilon} -(1/2-\epsilon)|\right] &= 2\epsilon P[\beta[k] > 1/2] \label{eq:AA-baz} \\
\mathbb{E}_{(1/2+\epsilon)}\left[|A[k]_{1/2+\epsilon} -(1/2+\epsilon)|\right] &= 2\epsilon P[\beta[k]\leq1/2] \label{eq:BB-baz} 
\end{align}
Thus for each $k \in \{1, 2, 3, \ldots\}$ such that $2\epsilon \leq \frac{\sqrt{3}}{4\sqrt{2k}}$ we have 
\begin{align}
&\mathbb{E}_{(1/2-\epsilon)}\left[R[k] - \theta^*_{(1/2-\epsilon)}T[k]\right] + \mathbb{E}_{(1/2+\epsilon)}\left[R[k] - \theta^*_{(1/2+\epsilon)}T[k]\right] \nonumber\\
&\overset{(a)}{\leq}\frac{-\epsilon}{4}\left[P[\beta[k]>1/2] + P[\beta[k]\leq 1/2]\right]\nonumber\\
&\overset{(b)}{=}\frac{-1}{8}\left(\mathbb{E}_{(1/2-\epsilon)}\left[|A[k]_{1/2-\epsilon} -(1/2-\epsilon)|\right] + \mathbb{E}_{(1/2+\epsilon)}\left[|A[k]_{1/2+\epsilon} -(1/2+\epsilon)|\right]\right)\nonumber\\
&\overset{(c)}{\leq} \frac{-\epsilon}{16} \label{eq:Baz-sum} 
\end{align}
where (a) holds by \eqref{eq:H-neg-case}-\eqref{eq:H-pos-case}; (b) holds by \eqref{eq:AA-baz}-\eqref{eq:BB-baz}; (c) holds by appciation of Lemma \ref{lem:abs-error} for $p=1/2-\epsilon$ and $q=1/2+\epsilon$. 
The condition $2\epsilon \leq \frac{\sqrt{3}}{4\sqrt{2k}}$ is equivalent to 
$$ k \leq \frac{3}{128 \epsilon^2}$$ 
Fix $K \in \{1, 2, \ldots, \lfloor \frac{3}{128 \epsilon^2} \rfloor\}$ (this corresponds to $K$ in the interval specified by Theorem \ref{thm:square-root-converse} with $\epsilon = 64\delta$). Summing \eqref{eq:Baz-sum} over $k \in \{0, \ldots, K\}$ gives 
\begin{align*}
&\sum_{k=0}^{K} \mathbb{E}_{(1/2-\epsilon)}\left[R[k] - \theta^*_{(1/2-\epsilon)}T[k]\right] + \sum_{k=0}^{K}\mathbb{E}_{(1/2+\epsilon)}\left[R[k] - \theta^*_{(1/2+\epsilon)}T[k]\right]\\
&\overset{(a)}{\leq}\sum_{k=1}^K \frac{-\epsilon}{16} \\
&= -\frac{K \epsilon}{16}
\end{align*}
where (a) neglects the nonpositive $k=0$ term (recall Lemma \ref{lem:y-theta-zero}). 
It follows that either 
\begin{align}
\sum_{k=0}^{K} \mathbb{E}_{(1/2-\epsilon)}\left[R[k] - \theta^*_{(1/2-\epsilon)}T[k]\right]  \leq -\frac{K\epsilon}{32}  \label{eq:either1} 
\end{align}
or 
\begin{align}
 \sum_{k=0}^{K}\mathbb{E}_{(1/2+\epsilon)}\left[R[k] - \theta^*_{(1/2+\epsilon)}T[k]\right] \leq -\frac{K\epsilon}{32}\label{eq:either2} 
\end{align}
Assume \eqref{eq:either1} holds: Then assuming $p=1/2-\epsilon$ and rearranging terms in  \eqref{eq:either1} yields (using notation $\expect{\cdot}$ instead of $\mathbb{E}_{(1/2-\epsilon)}\left[\cdot\right]$ for  simplicity): 
\begin{align*}
\frac{\sum_{k=0}^{K}\expect{R[k]}}{\sum_{k=0}^{K} \expect{T[k]}} &\leq \theta^* - \frac{K\epsilon}{32\sum_{k=0}^{K}\expect{T[k]}} \\
&\overset{(a)}{\leq} \theta^* - \frac{\epsilon}{64}\\
&\overset{(b)}{=} \theta^* - \delta
\end{align*}
where (a) holds because this example has $\expect{T[k]}\leq 2$ for all $k$; (b) holds because $\epsilon = 64\delta$.  Therefore, if \eqref{eq:either1} holds, then running the algorithm in the case when the true parameter is $p=1/2-\epsilon$ means that
\begin{equation} \label{eq:holds-or-not}
 \frac{\sum_{k=0}^K \expect{R[k]}}{\sum_{k=0}^K\expect{T[k]}} \leq \theta^*-\delta
 \end{equation} 
On the other hand, if \eqref{eq:either1} fails then \eqref{eq:either2} must hold
 and by the same argument it follows that \eqref{eq:holds-or-not} is true for the case
 $p=1/2+\epsilon$.   So for any particular causal algorithm,
  \eqref{eq:holds-or-not} must hold for either the case $p=1/2-\epsilon$ or the case $p=1/2+\epsilon$.  
 This proves Theorem \ref{thm:square-root-converse}. 
 
 \subsection{Discussion} \label{section:square-root-discussion} 
 
 Recall that Theorem \ref{thm:1} shows that the proposed
 algorithm of this paper achieves an optimality gap of $O(1/\sqrt{k})$ for general
 renewal optimization systems (including those without a strongly concave structure). 
 The square root converse result of 
 Theorem \ref{thm:square-root-converse} is theoretically important because
 it matches this achievability result and demonstrates that the $O(1/\sqrt{k})$ behavior cannot generally  be improved.  From the above result, one would expect a computer simulation of this system to converge more slowly when $p \approx 1/2$.  Surprisingly, simulations show the algorithm converges  quickly even for such cases (see Section \ref{section:sim}).  This may be due to the very small coefficient $\frac{3}{2^{19}}$ obtained in Theorem \ref{thm:square-root-converse}.  One interpretation is that, while Theorem \ref{thm:square-root-converse} proves that no mathematical analysis can improve convergence for general systems beyond a square root law, the small constant coefficient  suggests that performance is not significantly degraded  for practical scenarios and  timescales.

\section{Probability 1 convergence} \label{section:prob1}

This section considers the general system described in Section \ref{section:structure-assumptions} (not necessarily having the strongly concave property used in Section \ref{section:strongly-concave}). Recall that $\{S[k]\}_{k=0}^{\infty}$ is i.i.d. with some distribution $F_S(s)=P[S[k]\leq s]$ for $s \in \mathbb{R}^m$. The three theorems in this section show: 
\begin{itemize} 
\item (Theorem \ref{thm:fundamental-theta}) No algorithm can have a  sample path time average that exceeds $\theta^*$. 
\item (Theorem \ref{thm:theta-to-theta-star}) The algorithm  \eqref{eq:alg1a}-\eqref{eq:theta-update} with the stepsize rule $\eta[k]=\frac{1}{(k+2)T_{min}}$ ensures $\theta[k]\rightarrow \theta^*$ with probability 1. 
\item (Theorem \ref{thm:final-sample-path}) The algorithm  \eqref{eq:alg1a}-\eqref{eq:theta-update} with stepsize rule $\eta[k]=\frac{1}{(k+2)T_{min}}$ ensures
$$ \lim_{K\rightarrow\infty} \frac{\sum_{k=0}^{K-1} R[k]}{\sum_{k=0}^{K-1} T[k]} = \theta^* \quad \mbox{(with prob 1)} $$
\end{itemize}

\subsection{Preliminary lemma} 

Consider a probability experiment with sample space $\Omega$, sigma-algebra of events $\script{F}$, and probability measure $P:\script{F}\rightarrow\mathbb{R}$. For each $i\in\{1, 2, 3, \ldots\}$ let $X_i:\Omega\rightarrow\mathbb{R}$ be a random variable. 

\begin{lem} \label{lem:chow} 
Assume the sequence of random variables $\{X_i\}_{i=1}^{\infty}$ satisfies  
\begin{equation} \label{eq:chow-summable} 
\sum_{i=1}^{\infty} \frac{\expect{X_i^2}}{i^2} <\infty 
\end{equation} 
and there are finite constants $a,b$ such that with probability 1 we have 
\begin{equation} \label{eq:liminf-bounds} 
a\leq \expect{X_i|X_0, \ldots, X_{i-1}} \leq b \quad \forall i \in \{1, 2 ,3, \ldots\} 
\end{equation} 
where $X_0$ is defined to be $0$. 
Then with probability 1 we have 
\begin{equation} \label{eq:chow-main} 
 a \leq \liminf_{k\rightarrow\infty} \frac{1}{k}\sum_{i=1}^{k} X_i \leq \limsup_{k\rightarrow\infty} \frac{1}{k}\sum_{i=1}^{k} X_i\leq b
 \end{equation} 
\end{lem} 

\begin{proof} 
Define $Y_{0}=0$ and for each $i \in \{1, 2 ,3, \ldots\}$ define
\begin{equation} \label{eq:Y-def} 
Y_i = X_i - \expect{X_i|X_0, \ldots, X_{i-1}}
\end{equation}
By the law of iterated expectations we have 
\begin{equation} \label{eq:chow1} 
\expect{Y_i|Y_0, \ldots, Y_{i-1}} = 0 \quad \forall i \in \{0, 1, 2, \ldots\}
\end{equation} 
Further, since $(x-y)^2 \leq 2x^2 + 2y^2$ for all real numbers $x,y$, we have by \eqref{eq:Y-def} that for all positive integers $i$ 
$$ Y_i^2 \leq 2X_i^2 + 2\expect{X_i|X_0, \ldots, X_{i-1}}^2 $$
so for all positive integers $i$
\begin{align*}
\expect{Y_i^2} &\leq 2\expect{X_i^2}+ 2\expect{\expect{X_i|X_0, \ldots, X_{i-1}}^2} \\
&\overset{(a)}{\leq}2\expect{X_i^2} + 2\expect{\expect{X_i^2|X_0, \ldots, X_{i-1}}}\\
&\overset{(b)}{=} 4 \expect{X_i^2}
\end{align*}
where (a) holds by Jensen's inequality; (b) holds by the law of iterated expectations. 
Since \eqref{eq:chow-summable} holds it follows that 
\begin{equation} \label{eq:chow2} 
\sum_{i=1}^{\infty} \frac{\expect{Y_i^2}}{i^2} < \infty
\end{equation} 
By a theorem in \cite{chow-lln}, the conditions \eqref{eq:chow1} and \eqref{eq:chow2} imply
\begin{equation} \label{eq:chow-result}
 \lim_{k\rightarrow\infty} \frac{1}{k}\sum_{i=1}^k Y_i = 0 \quad \mbox{(with prob 1)}
 \end{equation}

On the other hand, by \eqref{eq:liminf-bounds} 
we have for all positive integers $k$ that with probability 1: 
$$a\leq \frac{1}{k}\sum_{i=1}^k \expect{X_i|X_0, \ldots, X_{i-1}} \leq b $$ 
Substituting the definition of $Y_i$ into the above inequality gives for all positive integers $k$ (with probability 1): 
$$ a\leq \frac{1}{k}\sum_{i=1}^k X_i - \frac{1}{k}\sum_{i=1}^k Y_i \leq b$$
Taking a $\limsup$ of the above inequality and using \eqref{eq:chow-result} gives (with probability 1): 
$$ a \leq \limsup_{k\rightarrow\infty} \frac{1}{k}\sum_{i=1}^k X_i \leq b$$
and a similar result holds when taking the $\liminf$.  This proves \eqref{eq:chow-main}. 
\end{proof}

\subsection{Sample path convergence} 

Recall that $\{S[k]\}_{k=0}^{\infty}$ are i.i.d. vectors in $\mathbb{R}^m$. 
Let $U$ be uniformly distributed over $[0,1)$ and independent of $\{S[k]\}_{k=0}^{\infty}$.  
Define 
\begin{align*}
W[0] &= 0 \\
W[k] &= (S[0], \ldots, S[k-1]) \quad \forall k \in \{1, 2, 3, \ldots\} 
\end{align*}
The variable $U$ is used as an external source of randomness to enable potentially randomized decisions.  As described in the previous section, there is no loss of generality in using a single variable $U$: From $U$ we can create an infinite sequence of i.i.d. uniformly distributed variables, all deterministic functions of $U$. 
A general \emph{causal and measurable} algorithm makes decisions on each frame $k \in \{0, 1, 2, \ldots\}$ by 
\begin{equation} \label{eq:sample-path1}
(T[k], R[k]) = f_k(U,W[k], S[k])
\end{equation} 
where for each $k$, $f_k(u,w,s)$ is a Borel-measurable function that satisfies 
\begin{equation} \label{eq:sample-path2} 
f_k(u,w,s) \in \script{D}(s) \quad \forall s \in \Omega_S
\end{equation} 

\begin{thm} \label{thm:fundamental-theta}  Under any causal and measurable algorithm we have 
$$ \limsup_{K\rightarrow\infty} \frac{\sum_{k=0}^{K-1}R[k]}{\sum_{k=0}^{K-1} T[k]} \leq \theta^* \quad (\mbox{with prob 1}) $$
where $\theta^*$ is from \eqref{eq:theta-star}. Furthermore
\begin{equation}  
T_{min} \leq \limsup_{K\rightarrow\infty}\frac{1}{K}\sum_{k=0}^{K-1}T[k]\leq T_{max} \quad \mbox{(with prob 1)} \label{eq:postpone-thm}
\end{equation} 
\end{thm} 

\begin{proof} 
Let $U$ be the uniform random variable and for each $k$ let $(T[k], R[k])$ be the
decisions from some causal and measurable algorithm 
(with structure defined by \eqref{eq:sample-path1}-\eqref{eq:sample-path2}). Fix $k \in \{0, 1, 2, \ldots\}$.  Since $S[k]$ is independent of $(U,W[k])$ we have from  
Theorem \ref{thm:measure-theory} in the Appendix: 
\begin{equation} \label{eq:in-view}
\expect{(T[k], R[k]) | U, W[k]} \in \overline{\script{A}}  \quad \mbox{(with prob 1)}
\end{equation} 
Furthermore from Theorem \ref{thm:measure-theory} and \eqref{eq:moment4}-\eqref{eq:moment5} we have 
\begin{align}
\expect{T[k]^2 | U, W[k]} &\leq C_1 \quad \mbox{(with prob 1)} \label{eq:in-view2} \\
\expect{R[k]^2|U, W[k]} &\leq C_2 \quad \mbox{(with prob 1)} \label{eq:in-view3} 
\end{align}
Define 
\begin{equation} \label{eq:Xk} 
X[k] = R[k] - \theta^*T[k] 
\end{equation} 

{\bf Claim 1:} There is a constant $V$ such that for all $k \in \{0, 1, 2, \ldots\}$ we have 
\begin{align*}
\expect{T[k]^2} &\leq V\\
\expect{R[k]^2} &\leq V\\
\expect{X[k]^2} &\leq V
\end{align*}

{\bf Proof of Claim 1:}  Fix $k \in \{0, 1, 2, \ldots\}$.  Taking expectations of \eqref{eq:in-view2} and using the law of iterated expectations gives
\begin{equation} \label{eq:C1-use} 
\expect{T[k]^2} \leq C_1
\end{equation} 
Similarly, taking expectations of \eqref{eq:in-view3} gives
\begin{equation} \label{eq:C2-use}
 \expect{R[k]^2} \leq C_2
 \end{equation} 
We have from the definition of $X[k]$: 
\begin{align*}
X[k]^2 &=(R[k]-\theta[k]T[k])^2\\
&\overset{(a)}{\leq} 2R[k]^2 + 2\theta[k]^2T[k]^2\\
&\overset{(b)}{\leq} 2R[k]^2 + 2\max\{\theta_{min}^2, \theta_{max}^2\} T[k]^2
\end{align*}
where (a) holds because $(x+y)^2\leq 2x^2 + 2y^2$ for any real numbers $x,y$; (b) holds
because $\theta[k] \in [\theta_{min}, \theta_{max}]$. Taking expectations
of the above and using \eqref{eq:C1-use}-\eqref{eq:C2-use} gives
\begin{equation} \label{eq:C3-use} 
\expect{X[k]^2} \leq 2C_2 + 2\max\{\theta_{min}^2, \theta_{max}^2\}  C_1
\end{equation} 
Defining $V$ as the maximum of the bounds in \eqref{eq:C1-use},\eqref{eq:C2-use},\eqref{eq:C3-use} proves Claim 1. 

 {\bf Claim 2:}  For all $k \in \{0, 1, 2, \ldots\}$ we have 
 \begin{align*}
& T_{min}\leq \expect{T[k] | T[0], ..., T[k-1]} \leq T_{max}  \quad \mbox{(with prob 1)}  \\
 & \expect{X[k] | X[0], \ldots, X[k-1]} \leq 0 \quad \mbox{(with prob 1)}
  \end{align*}
 
 {\bf Proof of Claim 2:}  By \eqref{eq:moment2} and Lemma \ref{lem:y-theta-zero}
 we know that every $(t,r) \in \overline{\script{A}}$ satisfies $t\geq T_{min}$ and $r/t \leq \theta^*$.  In view of \eqref{eq:in-view}, it holds that 
 \begin{align*}
 &T_{min}\leq \expect{T[k]|U, W[k]} \leq T_{max} \quad \mbox{(with prob 1)} \\
 &\expect{X[k]|U,W[k]} \leq 0 \quad \mbox{(with prob 1)} 
 \end{align*}
 Further 
 \begin{align*}
 \sigma(U, W[k]) &= \sigma(U, S[0], S[1], \ldots, S[k-1]) \\
 &\supseteq \sigma(T[0], T[1], \ldots, T[k-1]) 
 \end{align*}
 where the final inclusion holds because $(T[0], T[1], \ldots, T[k-1])$ is a deterministic
 function of $U, S[0], S[1], \ldots, S[k-1]$. Thus, for any random variable $J$ (specifically using either $J=T[k]$ or $J=X[k]$) we have 
 $$ \expect{J|T[0], \ldots, T[k-1]} = \expect{\: \underbrace{\expect{J|U, W[k]}}_{inner}\:| T[0], \ldots, T[k-1] }$$
 so that if the inner expectation satisfies a lower and/or upper bound with probability 1, so does the outer expectation.   This proves Claim 2. 
 
Claims 1 and 2 ensure that $T[k]$ and $X[k]$ satisfy the requirements of Lemma 
\ref{lem:chow}, and so 
\begin{align}
&T_{min}\leq \liminf_{K\rightarrow\infty} \frac{1}{K}\sum_{k=0}^{K-1} T[k] \leq T_{max} \quad \mbox{(with prob 1)} \label{eq:Tmin-inf} \\
&\limsup_{K\rightarrow\infty} \frac{1}{K}\sum_{k=0}^{K-1} X[k] \leq 0 \quad \mbox{(with prob 1)} \label{eq:Xinf} 
\end{align}
The inequality \eqref{eq:Tmin-inf} proves \eqref{eq:postpone-thm}. 
Substituting the definition of $X[k]$ (from \eqref{eq:Xk}) into  \eqref{eq:Xinf} yields (with prob 1)
$$ \limsup_{K\rightarrow\infty}\left[ \left(\frac{1}{K}\sum_{k=0}^{K-1} R[k]\right) - \theta^*\left(\frac{1}{K}\sum_{k=0}^{K-1} T[k]\right) \right] \leq 0$$
and from \eqref{eq:Tmin-inf} this implies the result. 
\end{proof} 

\begin{thm} \label{thm:theta-to-theta-star} Under algorithm \eqref{eq:alg1a}-\eqref{eq:theta-update} with stepsize $\eta[k]=\frac{1}{(k+2)T_{min}}$ for all $k \in \{0, 1, 2, \ldots\}$ we have 
$$ \lim_{k\rightarrow\infty} \theta[k] = \theta^* \quad \mbox{(with prob 1)} $$
where $\theta^*$ is from \eqref{eq:theta-star}.
\end{thm} 

\begin{proof} 
Fix $\epsilon>0$. By the Markov/Chebyshev inequality we have 
$$ P[|\theta[k]-\theta^*|\geq \epsilon] \leq \frac{\expect{(\theta[k]-\theta^*)^2}}{\epsilon^2} \leq \frac{2b}{k\epsilon^2T_{min}^2} \quad \forall k \in \{1, 2, 3, \ldots\} $$
where the final inequality holds by \eqref{eq:decreasing1}.
It follows that 
$$ \sum_{i=1}^{\infty}P[|\theta[i^2]-\theta^*| \geq  \epsilon] \leq \frac{2b}{\epsilon^2T_{min}^2}\sum_{i=1}^{\infty} \frac{1}{i^2}<\infty$$
And so the Borel-Cantelli theorem ensures that $\{|\theta[i^2]-\theta^*|\geq \epsilon\}$ happens for an at most finite number of indices $i$ with probability 1.  Since $\epsilon>0$ was arbitrary, this implies 
\begin{equation} \label{eq:over-squares} 
 \lim_{i\rightarrow\infty} \theta[i^2] = \theta^* \quad \mbox{(with prob 1)}
 \end{equation} 
Every positive integer $k$ must be between two perfect squares $n_k^2$ and $(n_k+1)^2$: 
$$ n_k^2 \leq k < (n_k+1)^2$$
Then for all positive integers $k$ we have 
$$ |\theta[k]-\theta^*| \leq |\theta[k]-\theta[n_k^2]| +  |\theta[n_k^2]-\theta^*|$$
Taking a $\limsup$ of the above inequality and using \eqref{eq:over-squares} yields 
\begin{equation*}
\limsup_{k\rightarrow\infty} |\theta[k] -\theta^*| \leq  \limsup_{k\rightarrow\infty} |\theta[k]-\theta[n_k^2]| \quad \mbox{(with prob 1)} 
\end{equation*}
It suffices to show the right-hand-side of the above inequality is $0$ with probability 1. 
For each positive integer $k$ define 
$$G_k = \max_{i \in \{k^2, \ldots, (k+1)^2-1\}} \left\{(\theta[i]-\theta[k^2])^2\right\}$$
Fix $\epsilon>0$.  It suffices to show that, with probability 1,  $\{G_k>\epsilon\}$ occurs for at most finitely many positive integers $k$.  To show this, by the Borel-Cantelli theorem it suffices to show
\begin{equation}\label{eq:summable} 
\sum_{k=1}^{\infty} P[G_k>\epsilon] < \infty
\end{equation}
For each positive integer $k$ we have by the Markov  inequality: 
\begin{align}
P[G_k>\epsilon] &\leq \frac{\expect{G_k}}{\epsilon}\nonumber \\
&=\frac{1}{\epsilon}\expect{\max_{i\in \{k^2, \ldots, (k+1)^2-1\}}\left\{(\theta[i]-\theta[k^2])^2\right\}}\label{eq:suby} 
\end{align}
Let $V$ be a finite constant that satisfies
\begin{equation} \label{eq:V-moment} 
\expect{(R[i]-\theta[i]T[i])^2} \leq V \quad \forall i \in \{0, 1, 2, \ldots\} 
\end{equation} 
Such a value $V$ exists because $\theta[i] \in [\theta_{min}, \theta_{max}]$ always, and 
second moments of $R[i]$ and $T[i]$ are uniformly bounded for all $i$.  Observe by the update procedure for $\theta[j]$ in \eqref{eq:theta-update} we have for all $j\in \{0, 1, 2, \ldots\}$ 
\begin{align} 
|\theta[j+1]-\theta[j]| &\overset{(a)}{=} \left|[\theta[j] + \eta[j](R[j]-\theta[j]T[j])]_{\theta_{min}}^{\theta_{max}} - [\theta[j]]_{\theta_{min}}^{\theta_{max}}\right| \nonumber \\
&\overset{(b)}{\leq} \left| \eta[j](R[j]-\theta[j]T[j])\right| \label{eq:projection} 
\end{align} 
where (a) uses $\theta[j] \in [\theta_{min},\theta_{max}]$; (b) uses the fact that
the distance between the projections of two numbers onto a closed interval is less than or equal to the distance between the numbers. For all positive integers $k$ and all $i \in \{k^2, \ldots, (k+1)^2-1\}$: 
\begin{align*}
(\theta[i]-\theta[k^2])^2 &=\left(\sum_{j=k^2}^{i-1} (\theta[j+1]-\theta[j])\right)^2\\
&\leq \left(\sum_{j=k^2}^{i-1} |\theta[j+1]-\theta[j]| \right)^2\\
&\overset{(a)}{\leq}\left(\sum_{j=k^2}^{i-1}\left|\eta[j](R[j]-\theta[j]T[j])\right|   \right)^2\\
&\overset{(b)}{\leq} \frac{1}{(k^2+2)^2T_{min}^2}\sum_{j=k^2}^{(k+1)^2-1} \sum_{r=k^2}^{(k+1)^2-1}|R[j]-\theta[j]T[j]|\cdot |R[r]-\theta[r]T[r]|
\end{align*}
where (a) holds by \eqref{eq:projection}; (b) holds because $\eta[j]\leq \frac{1}{(k^2+2)T_{min}}$ for all $j \geq k^2$. Taking the maximum of the above inequality over all $i \in \{k^2, \ldots, (k+1)^2-1\}$ gives
$$ \max_{i \in \{k^2, \ldots, (k+1)^2-1\}}\{(\theta[i]-\theta[k^2])^2\} \leq \frac{1}{(k^2+2)^2T_{min}^2}\sum_{j=k^2}^{(k+1)^2-1} \sum_{r=k^2}^{(k+1)^2-1}|R[j]-\theta[j]T[j]|\cdot |R[r]-\theta[r]T[r]|$$
Taking an expectation and using the Cauchy-Schwarz inequality and \eqref{eq:V-moment} gives
\begin{align*}
\expect{\max_{i \in \{k^2, \ldots, (k+1)^2-1\}}\left\{(\theta[i]-\theta[k^2])^2\right\}} \leq \frac{V\left((k+1)^2-k^2\right)^2}{(k^2+2)^2T_{min}^2} = \frac{V(2k + 1)^2}{(k^2+2)^2T_{min}^2}
\end{align*}
Substituting this into \eqref{eq:suby} gives
$$ P[G_k>\epsilon] \leq \frac{1}{\epsilon} \frac{V(2k+1)^2}{(k^2+2)^2T_{min}^2}$$
which decays like $O(1/k^2)$ and so $\sum_{k=1}^{\infty} P[G_k>\epsilon] <\infty$. 
\end{proof} 

\begin{thm} \label{thm:final-sample-path} Under algorithm \eqref{eq:alg1a}-\eqref{eq:theta-update} with stepsize $\eta[k]=\frac{1}{(k+2)T_{min}}$ for all $k \in \{0, 1, 2, \ldots\}$ we have 
$$ \lim_{K\rightarrow\infty} \frac{\sum_{k=0}^{K-1}R[k]}{\sum_{k=0}^{K-1} T[k]} = \theta^* \quad \mbox{(with prob 1)} $$
where $\theta^*$ is from \eqref{eq:theta-star}.
\end{thm} 
\begin{proof} 
We have by the rule \eqref{eq:alg1a}-\eqref{eq:alg1b}:
\begin{equation} \label{eq:prob1-foo} 
 R[k]-\theta[k]T[k] \geq R^*[k] - \theta[k]T^*[k] \quad \forall k \in \{0, 1, 2, \ldots\}
 \end{equation} 
where $(T[k], R[k]) \in \script{D}(S[k])$ is the decision made by the algorithm and $(T^*[k], R^*[k]) \in \script{D}(S[k])$ is any alternative decision.  Fix $(t,r) \in \script{A}$, so that $(t,r)$ can be achieved as an expectation of $(T^*[k], R^*[k])$ on frame $k$
under some particular decision rule. Let $(T^*[k], R^*[k]) \in \script{D}(S[k])$ be the decision that is made based purely on observing $S[k]$,  independent of the past, and that satisfies $\expect{(T^*[k], R^*[k])} = (t, r)$.  Then $\{(T^*[k], R^*[k])\}_{k=0}^{\infty}$ is a sequence of i.i.d. vectors and so by the law of large numbers: 
\begin{equation} \label{eq:LLN}
\lim_{K\rightarrow\infty} \frac{1}{K}\sum_{k=0}^{K-1} (T^*[k], R^*[k]) = (t,r) \quad \mbox{(with prob 1)} 
\end{equation} 
Rearranging terms in \eqref{eq:prob1-foo} gives the following for all $k \in \{0, 1, 2, \ldots\}$: 
\begin{align*}
R[k] - \theta^* T[k] \geq R^*[k] - \theta^* T^*[k] + (\theta[k]-\theta^*)(T[k]-T^*[k])
\end{align*}
Summing over $k \in \{0, \ldots, K-1\}$ and dividing by $K$ gives
\begin{align}
&\frac{1}{K}\sum_{k=0}^{K-1} R[k] - \theta^*\frac{1}{K}\sum_{k=0}^{K-1} T[k] \nonumber \\
&\geq \frac{1}{K}\sum_{k=0}^{K-1} R^*[k] - \theta^*\frac{1}{K}\sum_{k=0}^{K-1} T^*[k] - \frac{1}{K}\sum_{k=0}^{K-1} (\theta[k]-\theta^*)(T[k]-T^*[k]) \nonumber \\
&\geq \frac{1}{K}\sum_{k=0}^{K-1} R^*[k] - \theta^*\frac{1}{K}\sum_{k=0}^{K-1} T^*[k] - \frac{1}{K}\sum_{k=0}^{K-1}|\theta[k]-\theta^*|(T[k]+T^*[k])\label{eq:temp} 
\end{align}

{\bf Claim:} We have with probability 1 we have:
\begin{equation}
\limsup_{K\rightarrow\infty} \frac{1}{K}\sum_{k=0}^{K-1} |\theta[k]-\theta^*|(T[k]+T^*[k]) =0  \label{eq:postpone} 
\end{equation}

We postpone the proof of \eqref{eq:postpone}.  Taking a $\liminf$ of \eqref{eq:temp} and 
substituting \eqref{eq:postpone} and \eqref{eq:LLN} gives, with probability 1: 
$$\liminf_{K\rightarrow\infty} \left[   \frac{1}{K}\sum_{k=0}^{K-1} R[k] - \theta^*\frac{1}{K}\sum_{k=0}^{K-1} T[k]\right] \geq r - \theta^* t  \quad \mbox{(with prob 1)} $$
This holds for all $(t,r) \in \script{A}$.  Taking a limit over a countably infinite sequence of points $(t_i, r_i) \in \script{A}$ that approach the point $(t^*,r^*)\in \overline{\script{A}}$ and using the fact that $\theta^*=r^*/t^*$ gives
\begin{equation} \label{eq:dude} 
\liminf_{K\rightarrow\infty} \left[   \frac{1}{K}\sum_{k=0}^{K-1} R[k] - \theta^*\frac{1}{K}\sum_{k=0}^{K-1} T[k]\right] \geq 0 \quad \mbox{(with prob 1)} 
\end{equation} 
This, together with \eqref{eq:postpone-thm}, proves that with probability 1: 
$$ \liminf_{K\rightarrow\infty} \frac{\sum_{k=0}^{K-1} R[k]}{\sum_{k=0}^{K-1} T[k]} \geq \theta^*$$
On the other hand, Theorem \ref{thm:fundamental-theta} ensures the $\limsup$ is less than or equal to $\theta^*$ with probability 1.  Thus, the limit is exactly $\theta^*$ (with prob 1). 

It remains to prove \eqref{eq:postpone} of the Claim.  We know $\theta[k]\rightarrow\theta^*$ with probability 1. Fix $\epsilon>0$. With probability 1 we know  $|\theta[k]-\theta^*|\leq \epsilon$ for all sufficiently large $k$ and so 
$$\limsup_{K\rightarrow\infty} \frac{1}{K}\sum_{k=0}^{K-1}|\theta[k]-\theta^*|(T[k]+T^*[k]) \leq \epsilon \limsup_{K\rightarrow\infty} \frac{1}{K}\sum_{k=0}^{K-1}(T[k]+T^*[k])  \quad \mbox{(with prob 1)} $$
Observe that both  $T[k]$ and $T^*[k]$ are from causal and measurable algorithms and so they both satisfy \eqref{eq:postpone-thm}. Thus
$$\limsup_{K\rightarrow\infty} \frac{1}{K}\sum_{k=0}^{K-1}|\theta[k]-\theta^*|(T[k]+T^*[k]) \leq \epsilon 2T_{max}  \quad \mbox{(with prob 1)} $$
This holds for all $\epsilon>0$ and so \eqref{eq:postpone} follows.
\end{proof} 

\section{Simulation} \label{section:sim} 

This section presents simulations of the proposed algorithm under the initial condition $\theta[0]=\theta_{min}$ and stepsize $\eta[k] = \frac{1}{(k+2)T_{min}}$. 

\subsection{System 1: Selecting one of multiple projects}

Consider the project selection problem of Section \ref{section:selection}. 
On each frame $k$ we receive $N[k]$ new potential projects, where $N[k] \in \{0, 1, 2, 3\}$  with $P[N[k]=i]=p_i$ and 
$$ p_0=0.1, p_1 = 0.9-p, p_2 = p/2, p_3 = p/2$$
where $p \in [0, 0.9]$ is a parameter varied in the simulations (larger values of $p$ yield stochastically more projects).  The decision set for
each frame $k$ is 
$$ \script{D}(S[k]) = \{(1,0), (T_1, R_1), \ldots, (T_{N[k]}, R_{N[k]})\}$$ 
where the $(1,0)$ option corresponds to working on no project for 1 time unit (and receiving no profit). Given that $N[k]=i$, the vectors $(T_j,R_j)$ for $j \in \{1, \ldots, i\}$ are generated independently with $T_j \sim Uniform([1,10])$ and $R_j=A_jT_j$ where $A_j \sim Unif([0,  50])$.  
The proposed algorithm uses $[\theta_{min}, \theta_{max}] = [0,50]$ and $T_{min}=1$. 

Fig. \ref{fig:SIM1a} compares
the proposed algorithm to the greedy strategy that chooses the project 
$j$ that maximizes the ratio $R_j/T_j$. The proposed algorithm has significant gains. 
Fig. \ref{fig:SIM1b} explores convergence of the time average for 
one sample path of the proposed
algorithm over $2000$ frames for various parameter values $p$. It is difficult to calculate $\theta^*$ analytically so the dashed horizontal lines in Fig. \ref{fig:SIM1b} are estimated values of $\theta^*$ obtained from 5000 independent runs. It is interesting to note that, considering only frames for which $N[k]\geq 1$, the proposed algorithm chooses to reject all offered projects a significant fraction of time: For parameters 
$p \in \{0, 0.3, 0.6, 0.9\}$ the conditional rejection  probabilities (given $N[k]\geq 1$)  
were  $0.46, 0.40, 0.33, 0.25$, respectively. 

Fig. \ref{fig:SIM1heuristic} compares the proposed algorithm to the $\theta$-empirical 
heuristic algorithm of \cite{renewal-opt-tac} (see Section VI.b of \cite{renewal-opt-tac}). 
The $\theta$-empirical heuristic has a similar structure that maximizes $R[k]-\theta[k]T[k]$ over all $(T[k], R[k])$ choices, but sets $\theta[k]$ to the empirical time average reward seen up to frame $k$. In 
\cite{renewal-opt-tac} and \cite{sno-text} it is shown that \emph{if this algorithm converges} then it converges to the optimal $\theta^*$, but no proof of convergence and no convergence time results are known. 
From simulation it yields very similar results, and even (slightly) faster convergence, in comparison to the proposed algorithm (compare the dashed curves to the solid curves of the same color in Fig. \ref{fig:SIM1heuristic}).  The advantage of the proposed algorithm is that it comes with a proof of convergence along with convergence time guarantees.

\begin{figure}[t]
   \centering
   \includegraphics[width=3in]{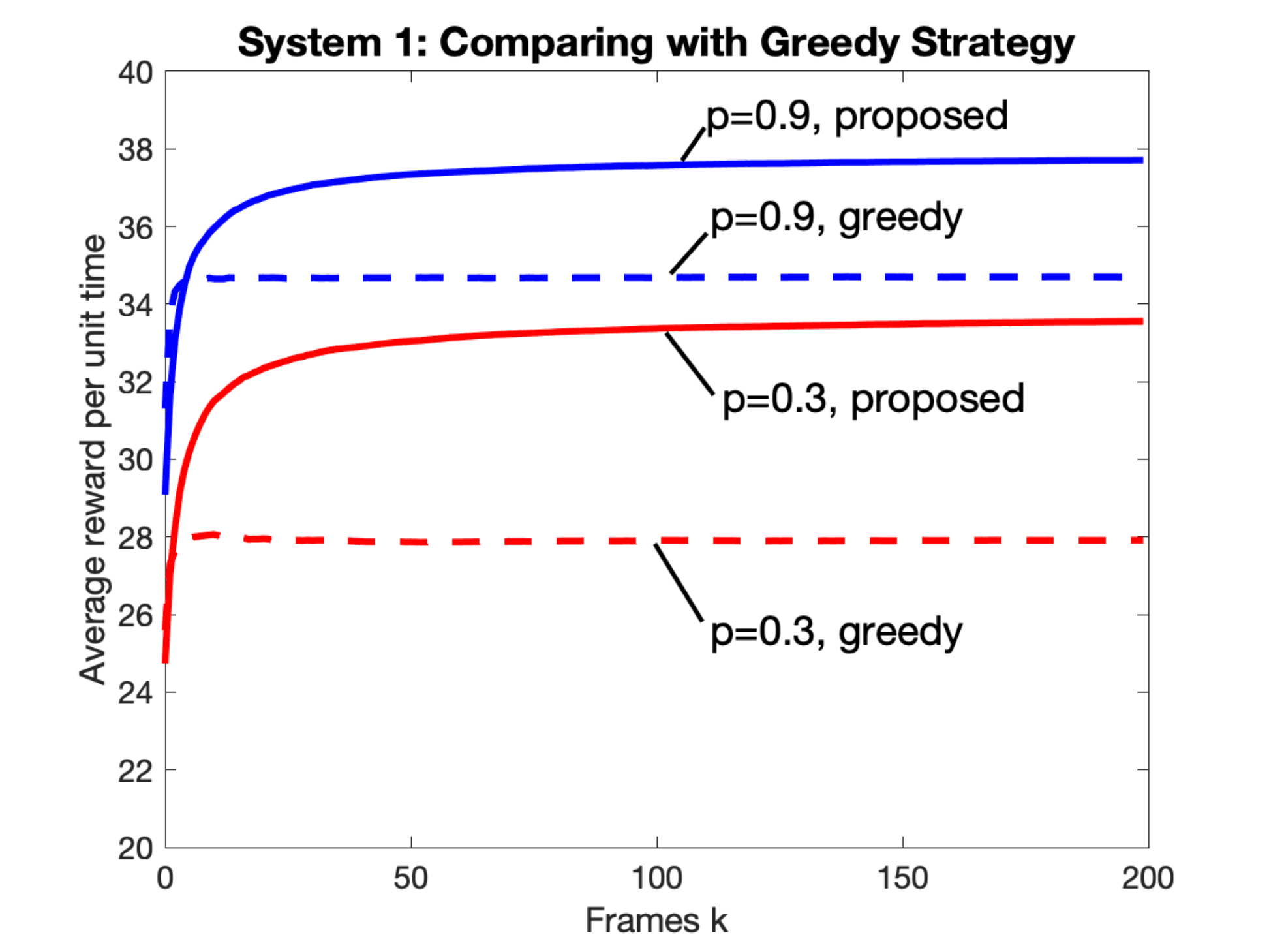} 
   \caption{Comparing the proposed algorithm with the greedy strategy for two different values of $p$. Data is averaged over 5000 independent sample paths.}
   \label{fig:SIM1a}
\end{figure}

\begin{figure}[h]
   \centering
   \includegraphics[width=3in]{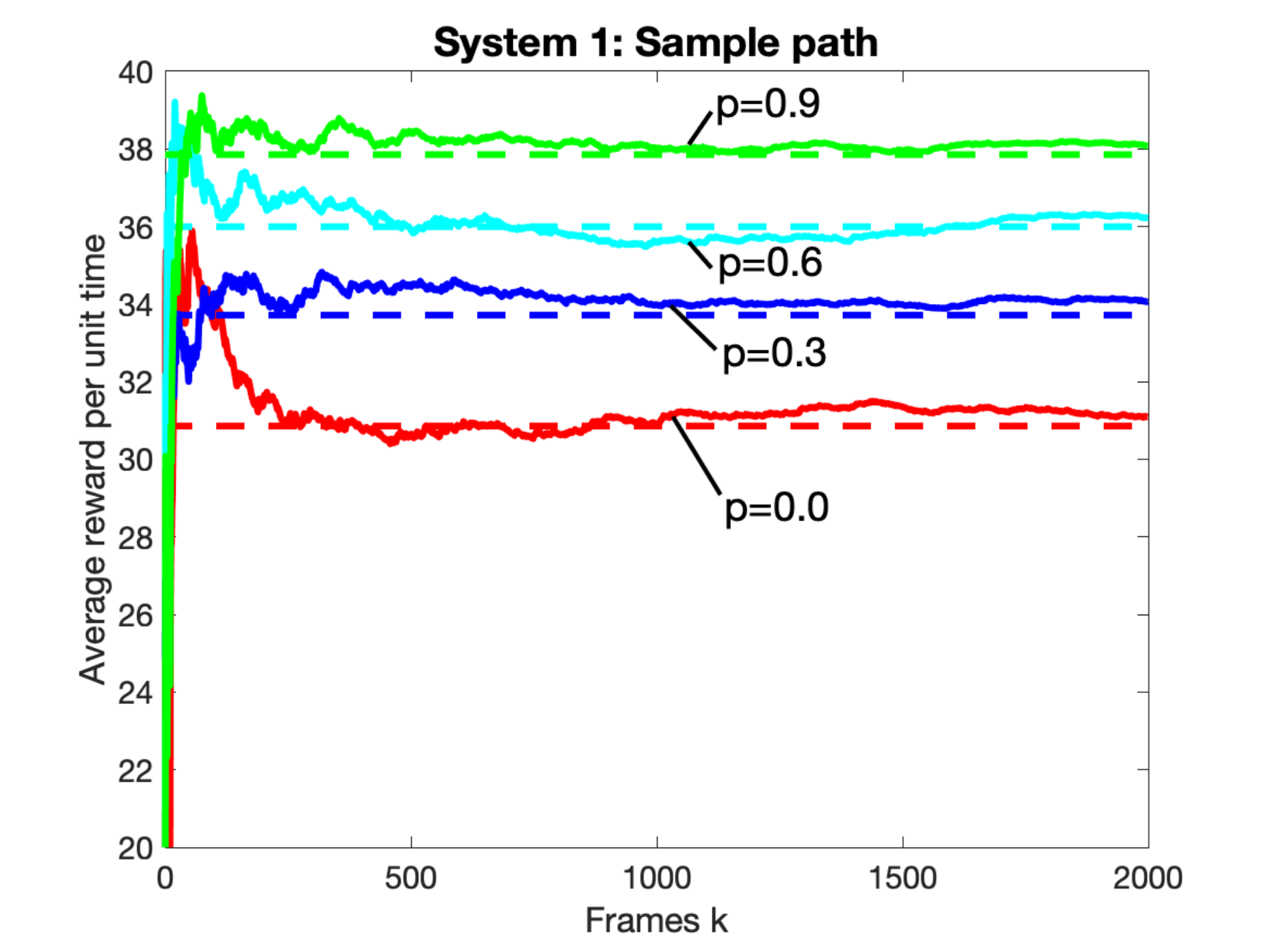} 
   \caption{Simulation of a single sample path for System 2: $\sum_{i=0}^{k-1} R[i]/\sum_{i=0}^{k-1}T[i]$ versus $k \in \{0, \ldots, 2000\}$ for four different values of $p$. Dashed horizontal lines are obtained by averaging the final value at time $2000$ over 5000 independent sample paths.}
   \label{fig:SIM1b}
\end{figure}

\begin{figure}[h]
   \centering
   \includegraphics[width=3in]{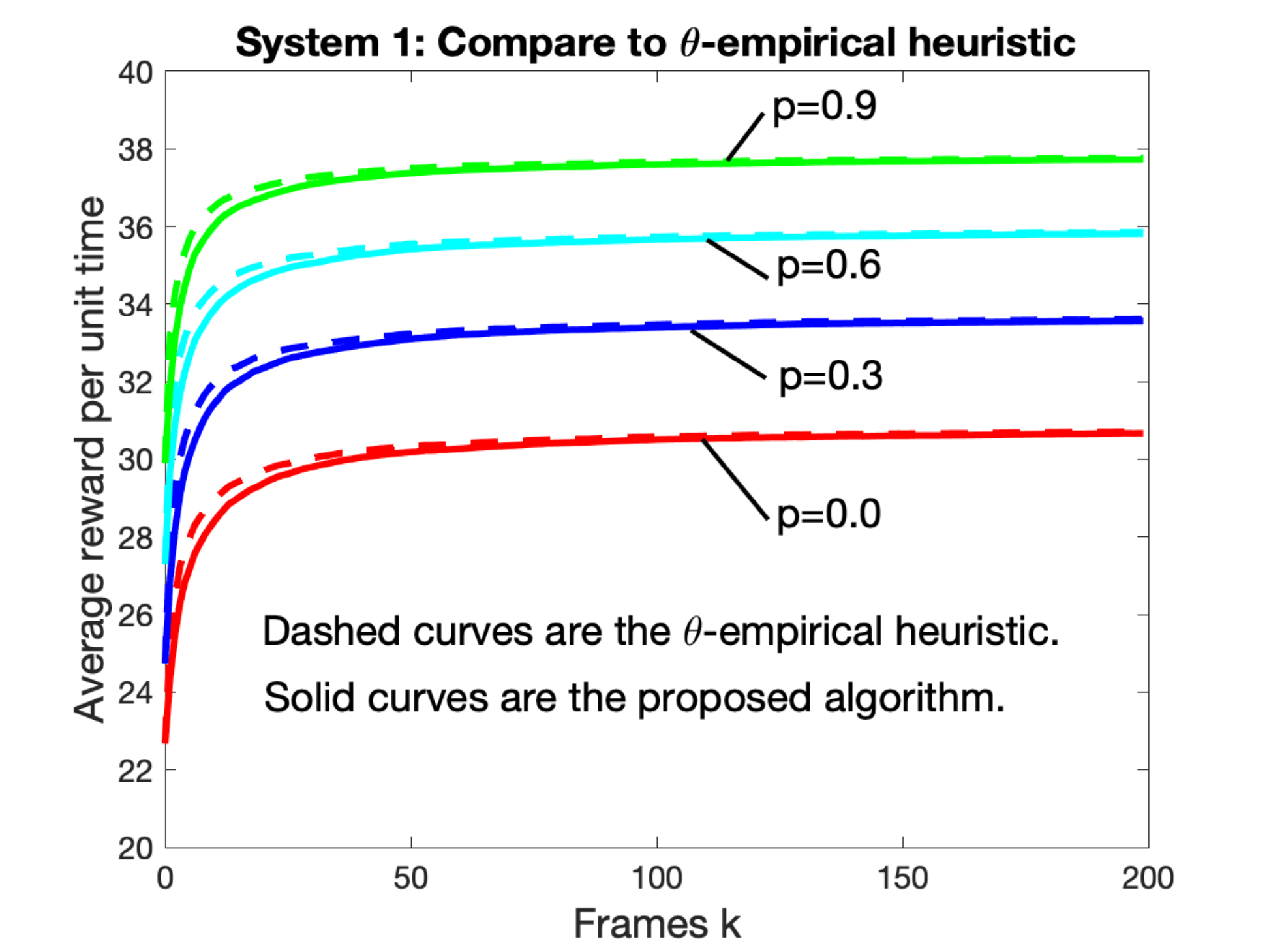} 
   \caption{Comparing the proposed algorithm to the $\theta$-empirical heuristic of \cite{renewal-opt-tac}.  The simulation averages over 5000 independent sample paths and plots $\mathbb{E}[\sum_{i=0}^{k-1} R[i]/\sum_{i=0}^{k-1}T[i]]$ versus $k \in \{0, \ldots, 200\}$ for four different values of $p$.}
   \label{fig:SIM1heuristic}
\end{figure}

\subsection{System 2: A curve of choices} 

Now consider the system described in Section \ref{section:system-converse}, where $\theta^*$ is analytically known.  The proposed algorithm uses $[\theta_{min}, \theta_{max}] = [1,2]$ and $T_{min}=1$. If $S[k]=1$ the algorithm chooses from a curve of choices:
$$ (T[k], R[k]) = (x[k], 2-(2-x[k])^2) $$
to maximize $T[k] - \theta[k]R[k]$ subject to $1\leq x[k]\leq 2$, which has 
solution  $x[k] = \left[2 - \frac{\theta[k]}{2}\right]_{1}^{2}$. 
Results for four different parameter values of $p=P[S[k]=1]$  
are in Figs. \ref{fig:SIM2a} and \ref{fig:SIM2b}.  The dashed horizontal lines are the analytically optimal $\theta^*$ values in \eqref{eq:example-theta}.  
Fig. \ref{fig:SIM2a} shows how close one sample path time average comes to $\theta^*$ after 1000 frames.  When the simulation time is extended it was observed that all four sample paths settled into near-constant values that were indistinguishable from $\theta^*$, which is consistent with the probability 1 sample path convergence proven in Section \ref{section:prob1}.  Fig. \ref{fig:SIM2b} shows the expected performance (with expectations computed by 
averaging over 5000 independent sample paths).  Fig. \ref{fig:SIM2b} plots $\expect{\sum_{i=0}^{k-1}R[i]/\sum_{i=0}^{k-1}T[i]}$. It was observed that plots of  $\sum_{i=0}^{k-1}\expect{R[i]}/\sum_{i=0}^{k-1}\expect{T[i]}$ over the same number of frames looked similar (those plots are omitted for brevity). 

\begin{figure}[htbp]
   \centering
   \includegraphics[width=3in]{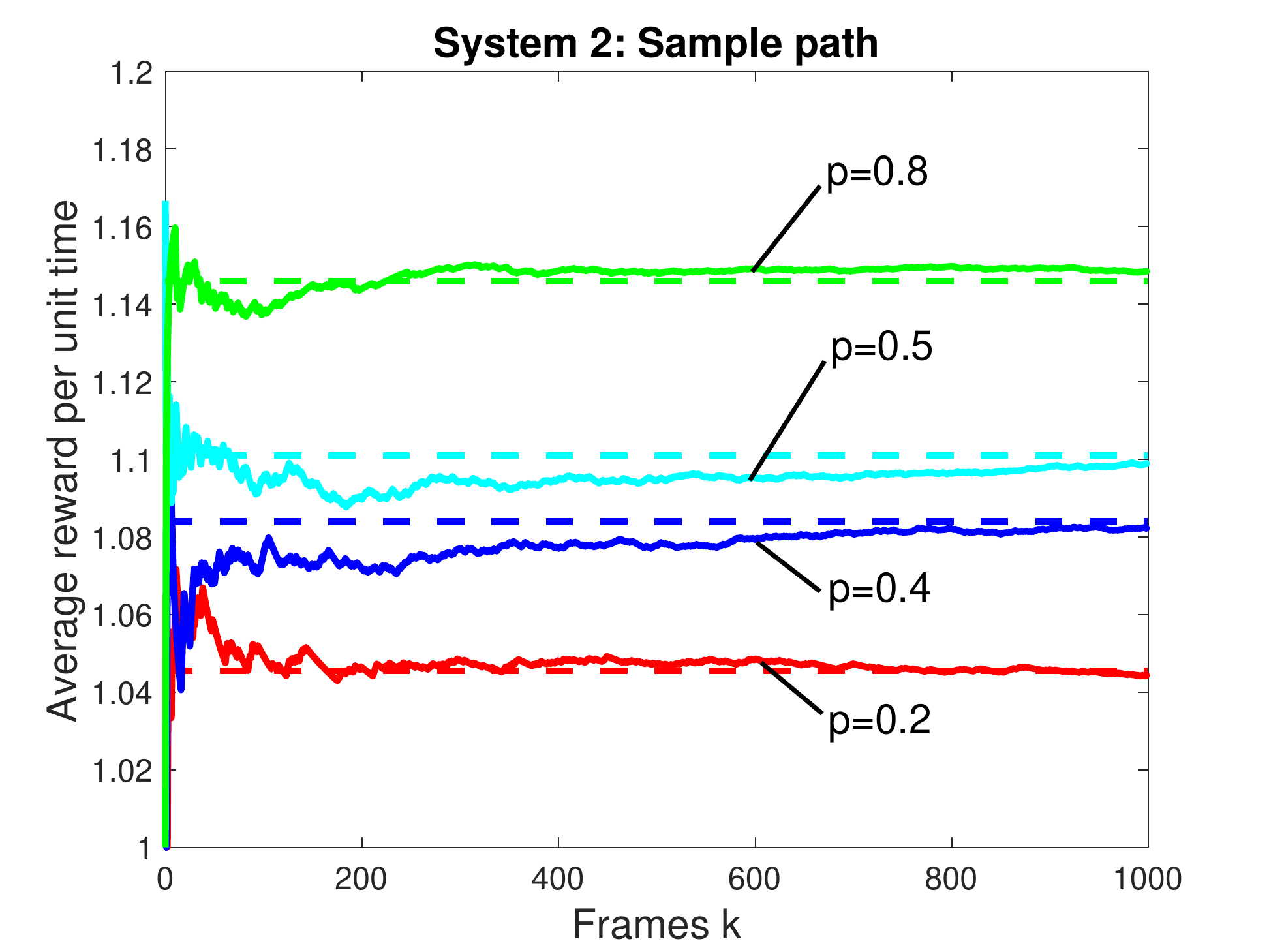} 
   \caption{Simulation of a single sample path for System 2: $\sum_{i=0}^{k-1} R[i]/\sum_{i=0}^{k-1}T[i]$ versus $k \in \{0, \ldots, 1000\}$ for four different Bernoulli probabilities $p$. Dashed horizontal lines are the optimal values.}
   \label{fig:SIM2a}
\end{figure}

\begin{figure}[htbp]
   \centering
   \includegraphics[width=3in]{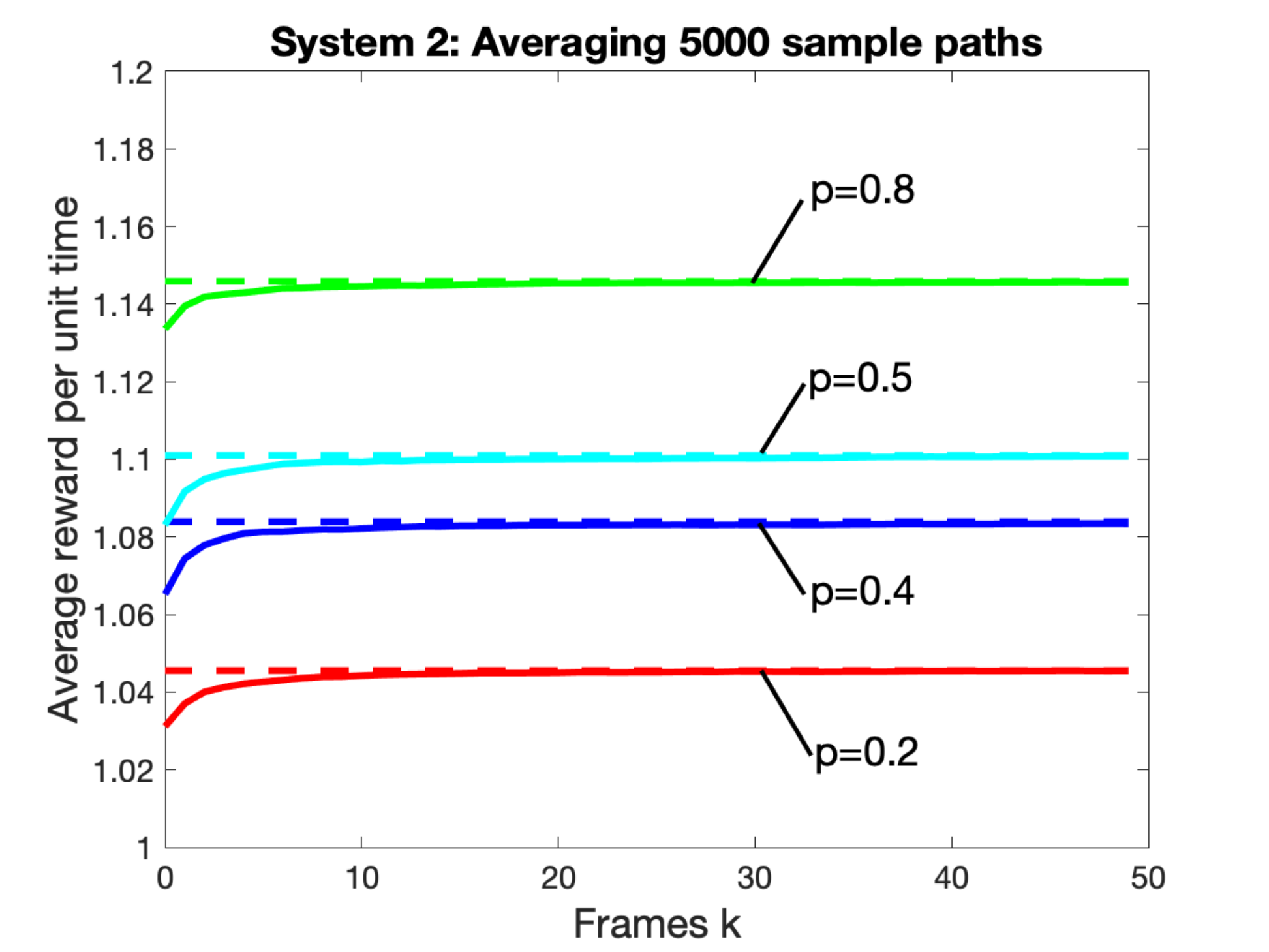} 
   \caption{Simulation averaging over 5000 sample paths for System 2: $\expect{\sum_{i=0}^{k-1} R[i]/\sum_{i=0}^{k-1}T[i]}$ versus $k \in \{0, \ldots, 50\}$ for four different Bernoulli probabilities $p$. Dashed horizontal lines are the optimal values.}
   \label{fig:SIM2b}
\end{figure}


\subsection{System 3: Two choices} 

Consider the system of Section \ref{section:example}, which is the same system for which a square-root converse result  was proven in Section \ref{section:square-root}.  When $S[k]=1$ there are only two choices: $(T[k], R[k]) \in \{(1,1), (2, 3)\}$. 
We use $[\theta_{min}, \theta_{max}] = [1, 3]$, $T_{min}=1$.  Fig. \ref{fig:SIM3a} plots data from a single sample path run over $1000$ frames 
for four different values of $p=P[S[k]=1]$.  The dashed horizontal lines are the exact $\theta^*$ values computed analytically in Section \ref{section:square-root}. Fig. \ref{fig:SIM3b} plots smoother curves that are averaged over $5000$ independent sample paths. The curves in Fig. \ref{fig:SIM3b} are plotted over the smaller timeline $k \in \{0, \ldots, 50\}$ to show convergence of the expected value. It can be seen that the algorithm converges quickly to optimality for all $p$ choices. There was no significant  behavioral difference observed when $p\approx 0.5$, even though the $p=0.5$  threshold played a crucial role in the square root converse result.\footnote{When one stares at  Fig. \ref{fig:SIM3b} long enough, one might become convinced of a \emph{very slight convergence time increase} for the green and black curves, representing data when $p\approx 0.5$, in comparison to the curves when $p$ is far from $0.5$. However, this difference is minor.  The author expected to see a bigger behavioral difference about $p=0.5$, but that did not occur. See discussion in Section \ref{section:square-root-discussion}.} 
Indeed, simulations were considered with 
$p=0.5 \pm \delta$ for various small $\delta$ values including $\delta=0$ (those curves fell in between the $p=0.49$ and $p=0.51$ curves of Fig. \ref{fig:SIM3b} but the data is omitted for clarity of the plots and for brevity). Our hypothesis about why the $p=0.5$ threshold was not more noticeable in the simulations is discussed
in Section \ref{section:square-root-discussion}.

\begin{figure}[htbp]
   \centering
   \includegraphics[width=3in]{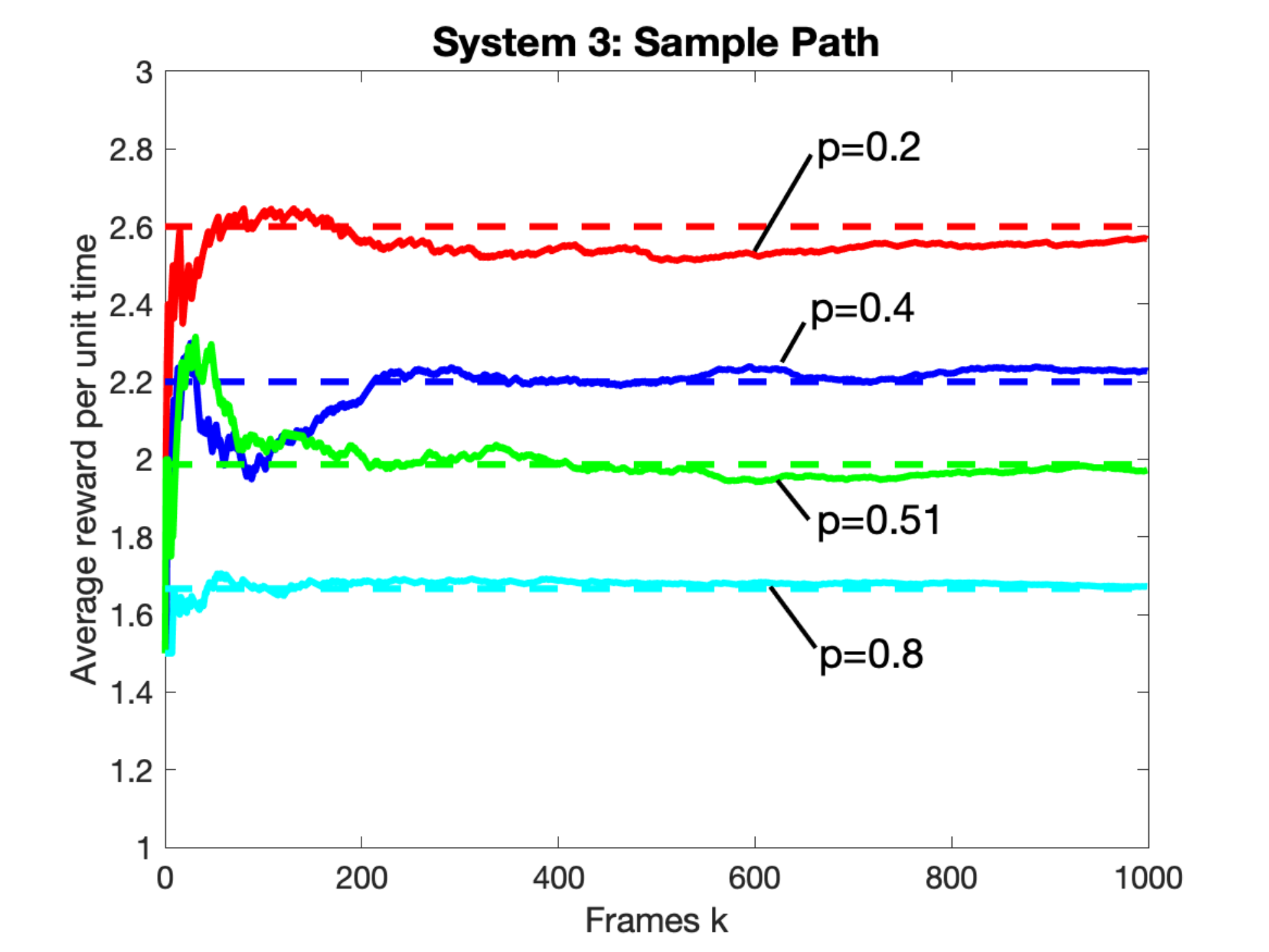} 
   \caption{Simulation of a single sample path for System 3: $\sum_{i=0}^{k-1} R[i]/\sum_{i=0}^{k-1}T[i]$ versus $k \in \{0, \ldots, 1000\}$ for four different Bernoulli probabilities $p$. Dashed horizontal lines are the optimal values.}
   \label{fig:SIM3a}
\end{figure}

\begin{figure}[htbp]
   \centering
   \includegraphics[width=3in]{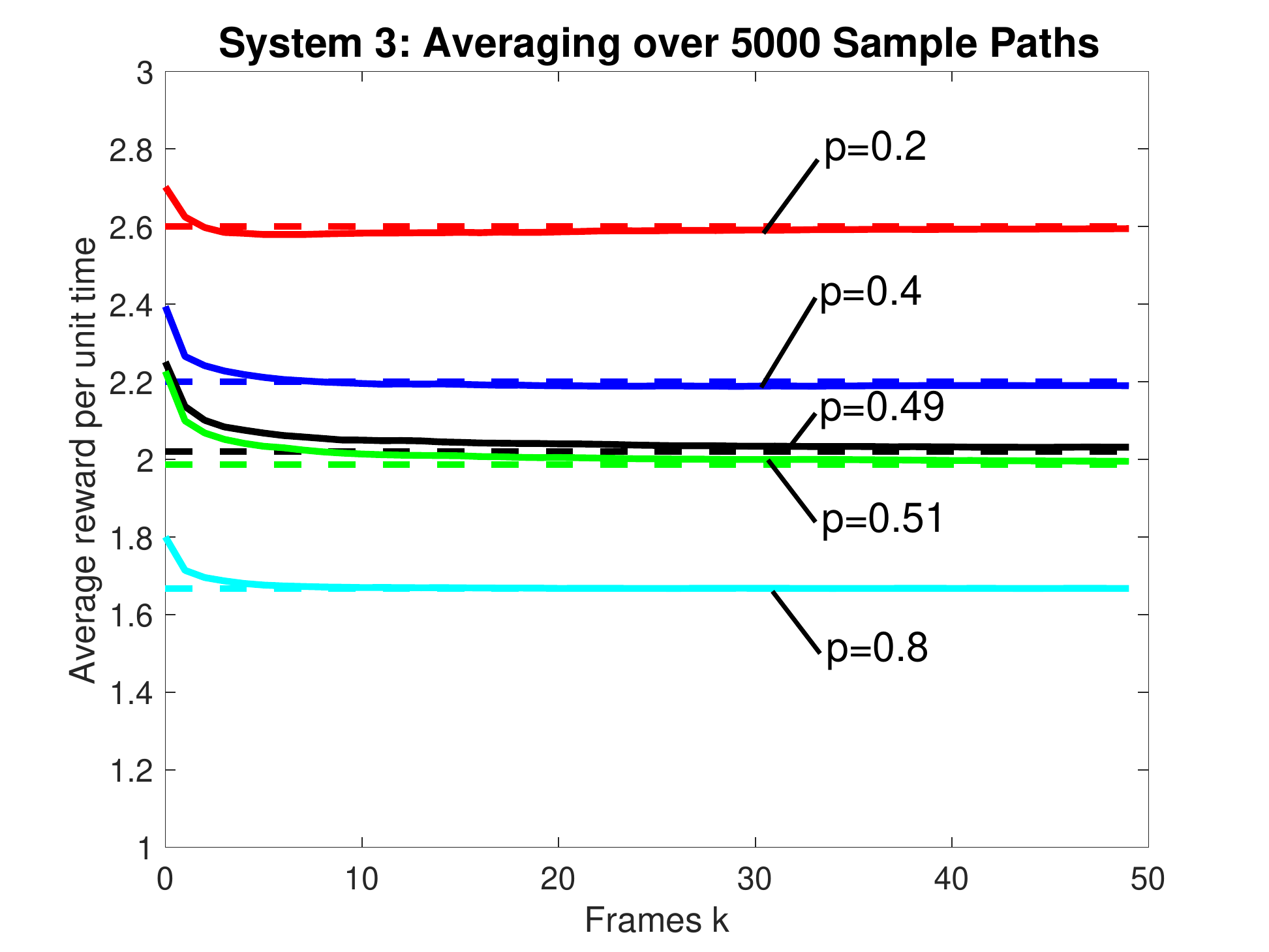} 
   \caption{Simulation averaging over $5000$ sample paths for System 3: $\expect{\sum_{i=0}^{k-1} R[i]/\sum_{i=0}^{k-1}T[i]}$ versus $k \in \{0, \ldots, 50\}$ for five different Bernoulli probabilities $p$. Dashed horizontal lines are the optimal values.}
   \label{fig:SIM3b}
\end{figure}

\section{Conclusion} 

This paper develops an online algorithm for making decisions in a renewal system. The algorithm
is shown to  converge 
to the optimal time average reward with the fastest possible asymptotic convergence time. 
The algorithm adjusts an auxiliary variable according to a 
Robbins-Monro iteration.  It also makes online decisions on each frame that are informed
by the current value of this variable.  When the system has a strongly concave structure the algorithm
is shown to have an optimality gap of $O(\log(k)/k)$. A matching converse result shows this
gap is the best possible in the strongly concave scenario.  In general conditions (without strong concavity)
the algorithm was shown to have an optimality gap of $O(1/\sqrt{k})$ and a matching
converse was also demonstrated.  The convergence results are presented in terms of expectations achieved by the algorithm. The algorithm was also shown to  have sample paths that converge to optimality with probability 1.

\section*{Appendix}

Consider a probability experiment with sample space $\Omega$, sigma-algebra of events $\script{F}$, and
probability measure $P:\script{F}\rightarrow\mathbb{R}$.  Fix $m, v$ as positive integers and consider random vectors 
\begin{align*}
&U: \Omega \rightarrow [0, 1)\\
&W: \Omega \rightarrow \mathbb{R}^v\\
&S: \Omega \rightarrow \mathbb{R}^m
\end{align*}
Assume that 
\begin{itemize} 
\item $U$ is uniformly distributed over $[0,1)$. 
\item $S$ is independent of $(U,W)$. 
\end{itemize} 
Define $\Omega_S$ and $\Omega_W$ by 
\begin{align*}
\Omega_S &= \{S(\omega) : \omega \in \Omega\}\\
\Omega_W &= \{W(\omega) : \omega \in \Omega\} 
\end{align*}
For each $s \in \Omega_S$ let $D(s)\subseteq\mathbb{R}^n$ be a general set.  We consider random vectors $X:\Omega\rightarrow\mathbb{R}^n$ of the form
$$ X=f(U,W,S)$$
where $f$ is a function that has the following properties, called \emph{$f$-requirements}:
\begin{itemize} 
\item $f:\mathbb{R}\times \mathbb{R}^v\times \mathbb{R}^m \rightarrow\mathbb{R}^n$ is Borel-measurable. 
\item $f(u,w,s) \in D(s)$ for all $u \in [0,1)$, $w \in \Omega_W$, $s \in \Omega_S$.
\end{itemize} 

Intuitively, we want to observe $S$ and choose a random vector $X \in D(S)$, where the vector $X$ is informed by observations of $(U,W)$.   The random vector $W$ 
represents any other parameters of the system that can be observed that are distinct from $S$. The random variable $U$ represents a source of external randomness that can be used to choose   $X$.

It is assumed throughout that there exists at least one function $f$  that satisfies the $f$-requirements. We want any random vector $X$ that is constructed this way to  
have a finite and bounded expectation.  For this, we additionally 
assume there is a nonnegative function $M:\Omega_S\rightarrow[0, \infty)$ such that 
\begin{equation} \label{eq:dominating} 
 \left(\sup_{x \in D(s)} \norm{x}\right) \leq M(s) \quad \forall s \in \Omega_S
 \end{equation} 
where $\norm{x}=\sqrt{\sum_{i=1}^n x_i^2}$, 
and such that $M(S)$ is a random variable
with finite expectation $\expect{M(S)}$.
 Let $A \subseteq \mathbb{R}^n$ be the set of all expectations $\expect{X}=(\expect{X_1}, \ldots, \expect{X_n})$ that can be achieved in this way, considering all possible functions $f$ that satisfy the $f$-requirements.    Let $\overline{A}$ be the closure of $A$. It can be shown that
 $\overline{A}$ is nonempty, compact, and convex. 

\begin{thm} \label{thm:measure-theory}  Let $X=f(U,W,S)$ were $f$ satisfies the $f$-requirements. Let $Z$ be a random variable that is a version of  $\expect{X|U, W}$. Then 
\begin{equation} \label{eq:suffices-as} 
 P\left[\:Z \in \overline{A}\: \right] = 1
 \end{equation} 
Further, there exists a random variable $\tilde{Z}$ that is surely in the set $\overline{A}$
and that is also a version of $\expect{X|U, W}$. 
\end{thm} 

\begin{proof} 
Let $Z$ be a version of $\expect{X|U,W}$.  Then 
\begin{equation} \label{eq:inclusions-start} 
 \sigma(Z) \subseteq \sigma((U,W)) 
 \end{equation} 
where $\sigma(Z)$ is the sigma-algebra generated by the random vector $Z$, and $\sigma((U,W))$ is the sigma-algebra generated by the random vector $(U,W)$.
The set $\overline{A}$ is closed, hence Borel-measurable, 
so  $P[Z \in \overline{A}]$ is well defined.  Suppose   \eqref{eq:suffices-as} fails, so that 
\begin{equation}\label{eq:less-than-1} 
P[Z \in \overline{A}]<1
\end{equation} 
We shall reach a contradiction. Since $\overline{A}$ is compact and convex it is \emph{constructible}, meaning
it is the \emph{countable} intersection 
of closed half-spaces (see Proposition 7.5.6 in 
\cite{constructible-book}):
\begin{equation} \label{eq:countable-intersection}
 \overline{A} = \cap_{j=1}^{\infty} \{x \in \mathbb{R}^n : a_j^{\transpose}x \leq b_j\}
 \end{equation} 
for some $a_j \in \mathbb{R}^k$ and $b_j \in \mathbb{R}$ for all $j \in \{1, 2, 3, \ldots\}$.  Thus
$$ P[Z \in \overline{A}] = P[\cap_{j=1}^{\infty} \{a_j^{\transpose}Z\leq b_j\}]$$
and so from \eqref{eq:less-than-1}: 
\begin{equation}\label{eq:exists-half-space}
P[\cap_{j=1}^{\infty} \{a_j^{\transpose}Z\leq b_j\}] < 1
\end{equation} 
The above inequality implies it is \emph{impossible} to have $P[a_j^{\transpose} Z \leq b_j]=1$ for all $j \in \{1, 2, 3, \ldots\}$ (since
the intersection of a countably infinite number of probability-1 events is again a probability-1 event). So there 
is at least one positive integer $k$ such that 
\begin{equation} \label{eq:j-star} 
P[a_{k}^{\transpose} Z\leq b_{k}]<1
\end{equation} 
For simplicity of notation define $a=a_{k}$ and $b=b_{k}$. With this notation \eqref{eq:j-star} becomes 
$P[a^{\transpose}Z \leq b]<1$ and so 
\begin{equation} \label{eq:v} 
P[a^{\transpose}Z > b]>0
\end{equation} 
Define $B\subseteq \Omega$ as the following event: 
\begin{equation} \label{eq:B-event} 
 B = \{a^{\transpose} Z> b\}
 \end{equation} 
By construction of $B$ we have  
$$ B \in \sigma(Z) \subseteq \sigma((U, W))$$
where the final inclusion holds by \eqref{eq:inclusions-start}. 
By  \eqref{eq:v} we know 
\begin{equation} \label{eq:Bgreater} 
P[B] > 0
\end{equation}  
Since $Z$ is a version of $\expect{X|U, W}$, we know that $a^{\transpose}Z$ is a version of $\expect{a^{\transpose} X |U, W}$. Since $B \in \sigma((U, W))$ we have by the defining property of a conditional expectation of $a^{\transpose}X$ given $(U,W)$: 
\begin{equation} \label{eq:pre-claim} 
\expect{ a^{\transpose} X 1_{B}}=\expect{a^{\transpose} Z 1_B} 
\end{equation} 
where $1_{B}$ is an indicator function for event $B$. 

{\bf Claim:} $\expect{a^T Z 1_B} > bP[B]$. 

{\bf Proof of Claim:} By definition of event $B$ in \eqref{eq:B-event} we have  
$$ \{ (a^T Z - b)1_{B}>0\} \iff B   $$
and so 
$$P[(a^T Z - b)1_{B}>0]=P[B] >0$$
where the final  inequality holds by \eqref{eq:Bgreater}.  Since the nonnegative
random variable $(a^{\transpose}Z-b) 1_{B}$ has a nonzero probability of being  
positive, it must have positive expectation. This proves the claim. 

Substituting the result of the claim into \eqref{eq:pre-claim} gives:
\begin{align*}
\expect{a^{\transpose}X 1_{B}}>bP[B] 
\end{align*}
Dividing by the positive number $P[B]$  gives
\begin{equation*} 
a^{\transpose} \left(\frac{\expect{X 1_{B}}}{P[B]}\right) > b 
\end{equation*} 
Since $P[B]>0$, the above inequality is equivalent to 
\begin{equation} \label{eq:contra}
 a^{\transpose} \expect{X|B} > b
 \end{equation} 

We now construct another random variable $\tilde{X}=\tilde{f}(U,W,S)$ 
for some other function $\tilde{f}$ that meets the $f$-requirements.  For any such $\tilde{X}$ we have by definition of the set $A$:
$$ \expect{\tilde{X}} \in A \subseteq \overline{A} \subseteq \{x \in \mathbb{R}^n : a^{\transpose} x \leq b\} $$
where the final inclusion is from \eqref{eq:countable-intersection}, 
and this directly implies 
\begin{equation} \label{eq:for-contradiction} 
a^{\transpose} \expect{\tilde{X}} \leq b
\end{equation} 
We can reach a contradiction between \eqref{eq:for-contradiction} and \eqref{eq:contra}   if the expectation of $\expect{\tilde{X}}$ can be chosen
 to match $\expect{X|B}$. 
 
 This can be done as follows: Since $P[B]>0$ we can define 
 $$F_{U, W|B}(u, w) = P[U \leq u, W\leq w| B] \quad \forall  (u,w) \in \mathbb{R}\times \mathbb{R}^v$$
 where vector 
 inequality $W \leq w$ is taken entrywise.   Map the 
uniformly distributed variable $U$  through a deterministic and  Borel-measurable function $\beta:\mathbb{R}\rightarrow \mathbb{R}^{v+1}$ to obtain a random vector 
$$(\alpha(U), \beta(U)) = (\alpha(U), \beta_1(U), \ldots, \beta_v(U))$$ 
that has distribution $F_{U,W|B}(u, w)$ for all $u \in \mathbb{R}$
and $w \in \mathbb{R}^v$. Define the function $\tilde{f}$ by 
$$ \tilde{f}(u, w,s) = f(\alpha(u), \beta(u), s) \quad \forall u \in \mathbb{R}, w \in \mathbb{R}^v, s \in \mathbb{R}^m$$
Since $f$ is Borel-measurable and $\beta$ is Borel-measurable, we know $\tilde{f}$ is Borel-measurable. Further, for all $(u,w, s) \in [0,1)\times \Omega_W\times \Omega_S$ we have 
$$ \tilde{f}(u,w,s) = f(\alpha(u),\beta(u), s) \in D(s) $$
where the final inclusion holds because $f$ satisfies the second $f$-requirement.  Thus, $\tilde{f}$ satisfies both $f$-requirements. 

The random variable $\tilde{X}$ is defined by 
\begin{equation} \label{eq:X-tilde} 
\tilde{X} = \tilde{f}(U,W, S) =  f(\alpha(U), \beta(U), S) 
\end{equation} 
Since $S$ is independent of $(U, W)$ and hence independent of $(\alpha(U), \beta(U))$, 
the random vector $(\alpha(U),\beta(U), S)$ has the same distribution as 
the conditional distribution of $(U, W, S)$ given that event  $B$ occurs. Thus  
$$\expect{f(\alpha(U), \beta(U), S)} = \expect{f(U, W, S) | B}$$
Multiplying both sides of the above equality by $a^{\transpose}$ and substituting 
$\tilde{X}=f(\alpha(U), \beta(U), S)$ gives 
$$ a^{\transpose}\expect{\tilde{X}} = a^{\transpose} \expect{f(U, W, S) | B} > b$$
where the final inequality holds by \eqref{eq:contra}.
This contradicts  \eqref{eq:for-contradiction}. This completes the proof of  $P[Z \in \overline{A}] = 1$. 

Finally, to construct another version of $\expect{X|U,W}$ that is \emph{surely} in the set $\overline{A}$ (rather than just \emph{almost surely}), 
fix any vector $a \in \overline{A}$ and define $\tilde{Z}$ by 
$$ \tilde{Z} = \left\{\begin{array}{cc}
Z & \mbox{ if $Z \in \overline{A}$}\\
a & \mbox { else} 
\end{array}\right.$$
Then  $\tilde{Z} \in \overline{A}$ surely, and $\sigma(\tilde{Z}) \subseteq\sigma(Z) \subseteq\sigma((U,W))$.  But $\tilde{Z}$ differs from $Z$ only on a set of measure zero, and so  
$\tilde{Z}$ is also a version of $\expect{X|U,W}$.
\end{proof} 

Lemma \ref{lem:tech-intro} from Section \ref{section:tech-intro}  is a special case of 
Theorem  \ref{thm:measure-theory}: Fix  $k \in \{0, 1, 2, \ldots\}$, define $S=S[k]$, $W=\theta[k]$, $(X_1, X_2) = (T[k], R[k])$.  The algorithm makes decisions for $(X_1,X_2)$ as a function of $W$ and $S$ (without using a source of randomness $U$).

\bibliographystyle{unsrt}
\bibliography{../../../../../../latex-mit/bibliography/refs}

\begin{thebibliography}{10}

\bibitem{gallager}
R.~Gallager.
\newblock {\em Discrete Stochastic Processes}.
\newblock Kluwer Academic Publishers, Boston, 1996.

\bibitem{fractional-paper}
S.~Schaible.
\newblock Fractional programming.
\newblock {\em Zeitschrift fur Operations Research}, vol. 27, no. 1, pp. 39-54,
  Dec. 1983.

\bibitem{boyd-convex}
S.~Boyd and L.~Vandenberghe.
\newblock {\em Convex Optimization}.
\newblock Cambridge University Press, 2004.

\bibitem{fox-linear-fractional-mdp}
B.~Fox.
\newblock {M}arkov renewal programming by linear fractional programming.
\newblock {\em Siam J. Appl. Math}, vol. 14, no. 6, Nov. 1966.

\bibitem{renewal-opt-tac}
M.~J. Neely.
\newblock Dynamic optimization and learning for renewal systems.
\newblock {\em IEEE Transactions on Automatic Control}, vol. 58, no. 1, pp.
  32-46, Jan. 2013.

\bibitem{sno-text}
M.~J. Neely.
\newblock {\em Stochastic Network Optimization with Application to
  Communication and Queueing Systems}.
\newblock Morgan \& Claypool, 2010.

\bibitem{xiaohan-datacenter}
X.~Wei and M.~J. Neely.
\newblock Data center server provision: Distributed asynchronous control for
  coupled renewal systems.
\newblock {\em IEEE/ACM Transactions on Networking}, 25(5), Aug. 2017.

\bibitem{xiaohan-asynch-renewal}
X.~Wei and M.~J. Neely.
\newblock Asynchronous optimization over weakly coupled renewal systems.
\newblock {\em Stochastic Systems}, 8(3), Sept. 2018.

\bibitem{neely-power-chapter}
M.~J. Neely.
\newblock Low power dynamic scheduling for computing systems.
\newblock In F.~R. Yu, X.~Zhang, and V.~C.~M. Leung, editors, {\em Green
  Communications and Networking}, pages pp. 219--259. CRC Press, 2012.

\bibitem{chih-ping-delay-optimal-priority}
C.~Li and M.~J. Neely.
\newblock Solving convex optimization with side constraints in a multi-class
  queue by adaptive $c\mu$ rule.
\newblock {\em Queueing Systems}, vol. 77, pp. 331-372, 2014.

\bibitem{robbins-monro}
H.~Robbins and S.~Monro.
\newblock A stochastic approximation method.
\newblock {\em Annals of Mathematical Statistics}, 22(3):400--407, 1951.

\bibitem{stochastic-approx-book}
A.~Nemirovski and D.~Yudin.
\newblock {\em Problem Complexity and Method Efficiency in Optimization}.
\newblock Wiley-Interscience Series in Discrete Mathematics, John Wiley, 1983.

\bibitem{SGD-averaging}
B.~T. Polyak and A.~B. Juditsky.
\newblock Acceleration of stochastic approximation by averaging.
\newblock {\em SIAM Journal on Control and Optimization}, 30(4):838--855, 1992.

\bibitem{SGD-robust}
A.~Nemirovski, A.~Juditsky, G.~Lan, and A.~Shapiro.
\newblock Robust stochastic approximation approach to stochastic programming.
\newblock {\em SIAM Journal on Optimization}, 19(4):1574--1609, 2009.

\bibitem{kushner-stochastic-approx}
H.~J. Kushner and G.~Yin.
\newblock {\em Stochastic Approximation and Recursive Algorithms and
  Applications}.
\newblock Springer, 2003.

\bibitem{borkar-book}
V.~S. Borkar.
\newblock {\em Stochastic Approximation: A Dynamical Systems Viewpoint}.
\newblock Springer, 2008.

\bibitem{proximal-robbins-monro}
P.~Toulis, T.~Horel, and E.~M. Airoldi.
\newblock The proximal {R}obbins-{M}onro method.
\newblock {\em arXiv:1510.00967v4}, Feb. 2020.

\bibitem{SGD-linear-models}
P.~Toulis, E.~Airoldi, and J.~Rennie.
\newblock Statistical analysis of stochastic gradient methods for generalized
  linear models.
\newblock {\em Proc. 31st International Conference on Machine Learning},
  32(2):667--675, 2014.

\bibitem{robbins-monro-binary}
V.~R. Joseph.
\newblock Efficient {R}obbins-{M}onro procedure for binary data.
\newblock {\em Biometrika}, 91(2):461--470, June 2004.

\bibitem{robbins-monro-Bayesian}
S.~Mandt, M.~D. Hoffman, and D.~M. Blei.
\newblock Stochastic gradient descent as approximate {B}ayesian inference.
\newblock {\em Journal of Machine Learning Research}, 18:1--35, 2017.

\bibitem{bubeck2012regret}
S.~Bubeck and N.~Cesa-Bianchi.
\newblock Regret analysis of stochastic and nonstochastic multi-armed bandit
  problems.
\newblock In {\em Foundations and Trends in Machine Learning}, pages 1--122,
  January 2012.

\bibitem{adversarial-bandit}
P. Auer, N. Cesa-Bianchi, and Y. Freund and R. E. Schapire.
\newblock Gambling in a rigged casino: The adversarial multi-armed bandit
  problem.
\newblock {\em 36th Annual Symposium on Foundations of Computer Science}, pp.
  322-331, Nov. 1995.

\bibitem{auer-confidence-bandit}
P.~Auer.
\newblock Using confidence bounds for exploitation-exploration trade-offs.
\newblock {\em Journal of Machine Learning Research}, 3:397--422, 2002.

\bibitem{atilla-renewal-2019}
S.~Cayci, A.~Eryilmaz, and R.~Srikant.
\newblock Learning to control renewal processes with bandit feedback.
\newblock {\em Proc. ACM SIGMETRICS}, July 2019.

\bibitem{hazan-kale-stochastic}
E.~Hazan and S.~Kale.
\newblock Beyond the regret minimization barrier: an optimal algorithm for
  stochastic strongly-convex optimization.
\newblock {\em Journal of Machine Learning Research}, vol. 15 (July):pp.
  2489--2512, 2014.

\bibitem{neely-NUM-converse-infocom2020}
M.~J. Neely.
\newblock A converse result on convergence time for opportunistic wireless
  scheduling.
\newblock {\em Proc. IEEE INFOCOM}, 2020.

\bibitem{Frank-Wolfe}
S.~Bubeck.
\newblock Convex optimization: Algorithms and complexity.
\newblock {\em Foundations and Trends® in Machine Learning},
  8(3-4):231--—357, 2015.

\bibitem{neely-converse-NUM-arxiv}
M.~J. Neely.
\newblock A converse result on convergence time for opportunistic wireless
  scheduling.
\newblock {\em ArXiv technical report, arXiv:2001.01031v4}, March 2020.

\bibitem{chow-lln}
Y.~S. Chow.
\newblock On a strong law of large numbers for martingales.
\newblock {\em Ann. Math Statist}, vol. 38, no. 2, 1967.

\bibitem{constructible-book}
J.~M. Borwein and J.~D. Vanderwerff.
\newblock {\em Convex Functions: Constructions, Characterizations and
  Counterexamples}.
\newblock Cambridge University Press, 2010.

\end{thebibliography}
\end{document}